\documentclass[11pt,bezier]{article}
\usepackage[usenames,dvipsnames]{xcolor}
\usepackage{amsmath,amssymb,amsfonts,euscript,graphicx,amsthm,mathrsfs}
\usepackage[colorlinks=true,citecolor=blue,linkcolor=red]{hyperref}
\usepackage[nameinlink]{cleveref}
\usepackage[all]{xy}
\usepackage{multicol}
\usepackage{enumerate}
\usepackage{tikz}

\usepackage[top=3cm, bottom=3cm, left=2.5cm, right=3cm]{geometry}
  \newtheorem{The}{Theorem}[section]
  \newtheorem{Pro}[The]{Proposition}
  \newtheorem{Lem}[The]{Lemma}
  \newtheorem{Cor}[The]{Corollary}
  
  \newtheorem{Defs}[The]{Definitions}
  \newtheorem{Rem}[The]{Remark}
  
  \newtheorem{Examp}[The]{Example}

\newtheorem{Ques}[The]{Question}

\newcommand{\M}{\mathcal{M}}
\newcommand{\Z}{\mathcal{Z}}

\let\oldproofname=\proofname
\renewcommand{\proofname}{\textit{\rm\bf\oldproofname}}

\title{\bf\Large A generalization of  uniserial modules and  rings 
 \thanks {The research of the third author was in part supported by a 
 grant from IPM (No. 1402160411). This research is partially
carried out in the IPM-Isfahan Branch.}
  \thanks
{{\it Key Words}:  Weakly uniserial module, Weakly uniserial ring, Torsion-free abelian group }
\thanks {2020{ \it Mathematics Subject Classification}. 16D10, 16D70, 16D60, 16N20, 16S90. }}

\author {{\bf S. Shirzadi$^{\rm a},$}   {\bf R. Beyranvand$^{\rm a,}$\thanks{Corresponding author.}} {~\bf and  A. Moradzadeh-Dehkordi$^{\rm b,c}$} \\
{\small{ $^{{\rm a}}$ Department of Mathematics, Faculty of Science, Lorestan University,  Khorramabad, Iran}}\\
{\small{ $^{{\rm b}}$ Department of Science, Shahreza Campus, University of Isfahan, Iran}} \\
{\small{ $^{{\rm c}}$ School of Mathematics, Institute for Research in Fundamental Sciences
(IPM)}}\vspace{-1mm}\\ {\small{ P.O.Box : 19395-5746, Tehran,
Iran}}\\
   {\small{beyranvand.r@lu.ac.ir}}\vspace{-1mm}\\
   {\small{a.moradzadeh@shr.ui.ac.ir}}\vspace{-1mm}\\
   {\small{shirzadi.sa@fs.lu.ac.ir}}\vspace{-1mm}\\
   {\small{}}}
  \date{}

\begin{document}
\maketitle
\begin{abstract}
We introduce and study  a nontrivial generalization of uniserial modules and rings. A module  is called {\it weakly uniserial} if its submodules are comparable regarding embedding.  Also, a {\it right} (resp., {\it left}) {\it weakly uniserial} ring is a ring which is weakly uniserial as a right (resp., left) module over itself.  In this paper, in addition to providing the properties of weakly uniserial modules and rings, we show that every right $R$-module is weakly uniserial if and only if every 2-generated right  $R$-module is weakly uniserial, if and only if $R\cong M_{n}(D)$ where $D$ is a division ring. Then  it is determined which torsion-free abelian groups of rank $1$ are weakly uniserial. Finally,  when $R$ is a commutative principal ideal domain, the structure of  finitely generated weakly uniserial $R$-modules are completely determined.
\end{abstract}
 
\section{\hspace{-6mm}. Introduction}

Recall that an $R$-module $M$ is said to be {\it uniserial} if its submodules are linearly ordered by inclusion. An $R$-module $M$ is called {\it serial} if it is a direct sum of uniserial modules. A {\it left} (resp., {\it right}) {\it uniserial} ring is a ring which is uniserial as a left (resp., right) module. Also, a ring $R$ is called {\it uniserial} (resp., {\it serial}) if it is both a left and a right uniserial (resp., serial) ring.  Studying serial rings and modules has been done since many  years ago. Köthe was probably the pioneer in this field \cite{Kothe}. He proved that over an Artinian principal ideal ring (a special case of serial rings), every module is a direct sum of cyclic submodules. In addition, Nakayama in \cite[Theorem 17]{Nak1}, showed that every left (right)  $R$-module is a serial module provided that $R$ is an Artinian serial ring.  Skornyakov in \cite{Sko} proved the converse. To observe the major properties of serial rings and modules, we can refer to \cite{Wis}.  Recently, in \cite{BehR,BehQ}, Behboodi et al., introduced and studied the notions of virtually uniserial modules and almost uniserial modules as generalizations of uniserial modules. 
An $R$-module $M$ is said to be {\it almost uniserial} if any two non-isomorphic submodules of $M$ are linearly ordered by inclusion. A  {\it virtually simple module} is a module which 
is isomorphic to each its nonzero submodule.  Also an $R$-module $M$ is called {\it virtually uniserial} if for every finitely generated submodule $0 \neq N \subseteq M$, $N/Rad(N)$ is virtually simple.  \\
In this paper, we introduce and study the concept of  weakly uniserial modules  as a  nontrivial common  generalization of uniserial modules and almost uniserial modules.
Let $M$ be a right $R$-module. We say that $M$ is {\it weakly uniserial} if for any two submodules $N$ and $K$ of $M$, $N$ is embedded in $K$ or $K$ is embedded in $N$.  A {\it left} (resp., {\it right}) {\it weakly uniserial ring} is a ring which is weakly uniserial as a left (resp., right) $R$-module. As usual, a ring $R$ is called {\it weakly uniserial} if it is both a left and a right weakly uniserial ring. We note that every uniserial (almost uniserial) module is  weakly uniserial, but the converse is not true in general. For instance, every vector space with dimension at least $3$ over a division ring is a weakly uniserial module, but it is not almost uniserial.

Among the other basic properties of weakly uniserial modules and rings, we study the following questions:\\
{\bf Question 1:}~{\it Which rings R have the property that every  {\rm (}$2$-generated or injective{\rm )} right R-module is weakly uniserial?}\\
{\bf Question 2:}~{\it When is a torsion-free abelian group of rank $1$ a weakly uniserial group?}

Throughout this paper, all rings have identity elements and all modules are unitary right modules.
For a ring $R$, the Jacobson radical and the right singular  ideal of $R$ are denoted by
${\rm J}(R)$ and $ \Z(R_{R})$, respectively. For a module $M$, the socle, the
injective hull and the singular submodule of $M$ are denoted by ${\rm Soc}(M)$, ${\rm E}(M)$ and $\Z(M)$, respectively. Also we denote  the set of associated primes of $M$   by ${\rm Ass}(M)$.  For a subset $X$ of $R$, the right (resp., left) annihilator
of $X$ in $R$ is denoted by ${\rm r.Ann}_{R}(X)$ (resp., ${\rm l.Ann}_{R}(X)$). By $K \subseteq M$ we usually mean that $K$ is a submodule of $M$ and the notation $N \subseteq_{e} M $ means that $N$ is an essential submodule of $M$. The cardinal number of a set $X$ is denoted by $|X|$ and for any two $R$-modules $M, N$ if there exists an $R$-monomorphism from $M$ to $N$, we write $ M\rightarrowtail N$  otherwise we write $ M\not\rightarrowtail N$. 

The paper is organized as follows. In Section $2$, we give some  basic properties of weakly uniserial modules and rings. It is shown that being weakly uniserial is a Morita invariant property (see \Cref{Morita}). Like of uniform modules, every weakly uniserial module such as $M$ has $|{\rm Ass}(M)|\leq1$, while the class of uniform modules and the class of weakly uniserial modules are independent (see \Cref{Ass(M)}).  In \Cref{domain}, we prove that the class of weakly uniserial rings is a straightforward generalization of domains.
 For  a commutative principal ideal domain $R$ with  the field of fractions  $Q$, we show that if $R$ is local, then $Q$ is a {\rm(}weakly{\rm)} uniserial $R$-module and if  $R$ is not local, then $Q$ is not a weakly uniserial $R$-module (see \Cref{field of fractions}). In particular, $\Bbb{Q}$ (the field of rational numbers) is not weakly uniserial as a $\Bbb{Z}$-module.
 One of the important results in this section is \Cref{N and Socle} which says that if $M$ is a weakly uniserial module such that ${\rm Soc}(M)$ is Dedekind-finite, then $N \subseteq {\rm Soc}(M)$ or ${\rm Soc}(M) \subseteq N$, for any submodule $N$ of $M$. In Section $3$, we answer Question $1$ and completely determine the structure of rings $R$ for which every {\rm (}$2$-generated or injective{\rm )}  right $R$-module is weakly uniserial (see \Cref{main theorem}). 
 In Section $4$, we answer Question $2$ and prove that if  $A$ is a torsion-free abelian group of ${\rm rank}$ $1$, then $A$ is weakly uniserial if and only if ${\rm type}(A)=[(\alpha_{p})]$, where $\alpha_{p}=0$ for all but a finite number $p$ and there is at most one $p$ such that $\alpha_{p}=\infty$ (see \Cref{t.f.a.g}).
  In Section $5$,  it is  completely determined  the structure of  finitely generated weakly uniserial $R$-modules, where $R$ is a commutative principal ideal domain (see \Cref{fundamental}). In particular, it is shown that a finitely generated $\mathbb{Z}$-module $M$ is weakly uniserial if and only if $M \cong \mathbb{Z}^{n} $ or $M \cong \oplus_{n} \mathbb{Z}_{p}$ or $M \cong \mathbb{Z}_{p^{n}}$, where $p$ is a prime number and $n \geq 0$ is an integer. Also the structure of weakly uniserial $\mathbb{Z}$-modules with nonzero socle is given in \Cref{w.uni z-module}.

\section{\hspace{-6mm}. Weakly uniserial modules and rings}

In this section, we give some basic properties of weakly uniserial modules and rings.

\begin{Defs}\label{w-uniserial}
{\rm   Let $M$ be a right $R$-module. We say that $M$ is {\it weakly uniserial} if for any two submodules $N$ and $K$ of $M$, $ N\rightarrowtail K$ or $ K\rightarrowtail N$.  A {\it left} (resp., {\it right}) {\it weakly uniserial ring} is a ring which is weakly uniserial as a left (resp., right) module. As susal, a ring $R$ is called {\it weakly uniserial} if it is both a left and a right weakly uniserial ring.}
\end{Defs}

 First of all, the following examples show that the notions of weakly uniserial rings and modules are nontrivial generalizations of uniserial rings and modules.

\begin{Examp}\label{wuni and uni}
\rm{(1)} For any prime number $p$, the $\mathbb{Z}$-module $\mathbb{Z}_{p} \oplus \mathbb{Z}_{p}$ is weakly uniserial but is not uniserial.

\rm{(2)} Since each nonzero $\Bbb Z$-submodule of $\Bbb Z$ is isomorphic to $\Bbb Z$, the ring $\mathbb{Z}$ is weakly uniserial but is not uniserial.
 \end{Examp}

Recall that an $R$-module $M$ is said to be {\it almost uniserial} if any two non-isomorphic submodules of $M$ are linearly ordered by inclusion (see \cite{BehR}). Clearly uniserial and almost uniserial right $R$-modules are always weakly uniserial, but the following example shows that the converse is not true in general.

\begin{Examp}\label{wuni and almostuni}
{\rm Every vector space with dimension at least $3$ over a division ring is a weakly uniserial module, but it is not almost uniserial.}
 \end{Examp}

Recall that an $R$-module $M$ is said to be {\it virtually simple} if $M \neq 0$ and $M \cong N$ for every nonzero submodule $N$ of $M$ (see \cite{Beh, Veda}). Also, an $R$-module $M$ is called {\it virtually uniserial} if for every finitely generated submodule $0 \neq N \subseteq M$, $N/Rad(N)$ is virtually simple (see \cite{BehQ}). The following example shows that the class of weakly uniserial modules and the class of virtually uniserial modules are independent.

\begin{Examp}\label{wuni and virtuni}
\rm{(1)} As stated  in \Cref{wuni and almostuni}, every vector space with dimension at least $3$ over a division ring is weakly uniserial. But, it is not virtually uniserial by \cite[Proposition 2.4]{BehQ}.

\rm{(2)} By \cite[Example 2.1(1)]{BehQ}, $\mathbb{Q}$ is a virtually uniserial $\mathbb{Z}$-module, but it is not weakly uniserial (see \Cref{Q not weak}(a)).
 \end{Examp}

Consequently, we have the following relationships:
$$\{\rm Uniserial~ modules\} \subsetneq \{\rm Almost~ uniserial~ modules\}\subsetneq \{\rm Weakly~ uniserial~ modules\},$$

\begin{center}
\begin{tikzpicture} 
\filldraw (-2.0,4.5) circle node {$\{ \rm Weakly \ uniserial \ modules \}\ $};
\filldraw (3.5,4.5) circle node {$\{ \rm Virtually \ uniserial \ modules \}.\ $};
\filldraw (0.6,4.74) circle node {$\nsubseteq$};
\filldraw (0.6,4.31) circle node {$\nsupseteq$};
\end{tikzpicture} 
\end{center}

In the following we give some basic properties of weakly uniserial modules and rings.
\begin{Rem}\label{submodule}
\rm{(1)} Every submodule of a weakly uniserial module is weakly uniserial.

\rm{(2)} Every quotient of a weakly uniserial module is not necessarily weakly uniserial. For instance, $\mathbb{Z}$ is a weakly uniserial $\mathbb{Z}$-module, but $\mathbb{Z} \slash 6 \mathbb{Z}$ is not a weakly uniserial $\mathbb{Z}$-module.

\rm{(3)} The notion of weakly uniserial module is preserved under isomorphism.
\end{Rem}

\begin{Pro}\label{central idempotent}
For a ring $R$, the following statements hold:\\
\indent {\rm (a)} If $R$ is right {\rm(}left{\rm)} weakly uniserial, then $R$ has no nontrivial central idempotent.\\
\indent {\rm (b)} If $R \cong \prod_{i \in I} R_{i}$, where $ |I| \geq 2$ and each $R_{i}$ is a ring, then $R$ is neither 
a right nor a left \indent\indent weakly uniserial ring.
\end{Pro}

\begin{proof}
 (a). Assume that $e$ is a central idempotent of $R$. Thus  $R=eR \oplus (1-e)R$ and since $R$ is right weakly uniserial, $eR \rightarrowtail (1-e)R$ or $(1-e)R \rightarrowtail eR$. If $eR \rightarrowtail (1-e)R$, then 
 $ {\rm r.Ann}_{R}(1-e)R \subseteq {\rm r.Ann}_{R}(eR)$ and so $e=0$. If $(1-e)R \rightarrowtail eR$, then $ {\rm r.Ann}_{R}(eR) \subseteq {\rm r.Ann}_{R}(1-e)R$ and so $e=1$. The left is similarly.

 (b) follows from (a).
\end{proof}

The following result shows that the weakly uniserial property is Morita invariant.

\begin{Pro}\label{Morita}
Being weakly uniserial is a Morita invariant property.
\end{Pro}
\begin{proof}
Assume that $R$ and $S$ are Morita equivalent rings via inverse equivalences $F : Mod_{R} \rightarrow Mod_{S}$, $G : Mod_{S} \rightarrow Mod_{R}$, and let $\eta : FG \rightarrow 1_{Mod_{S}}$, $\xi : GF \rightarrow 1_{Mod_{R}}$ be the corresponding natural isomorphisms. Suppose that $M$ is a weakly uniserial right $R$-module. We show that $F(M)$ is a weakly uniserial right $S$-module. Assume that $A$ and $B$ are submodules of $F(M)$. Consider the inclusions $i_{A} : A \rightarrow F(M) $ and  $i_{B} : B \rightarrow F(M) $. Then by \cite[Proposition 21.2]{AF}, $G(A) \overset{G(i_{A})} \longrightarrow GF(M)$ and $G(B)\overset{G(i_{B})} \longrightarrow GF(M)$ are monomorphisms and so we have the monomorphisms $G(A) \overset{G(i_{A})} \longrightarrow GF(M) \overset{\xi} \longrightarrow M$, $G(B) \overset{G(i_{B})} \longrightarrow GF(M) \overset{\xi} \longrightarrow M$. Thus $G(A) \cong \xi G(i_{A})(G(A)) \subseteq M$ and $G(B) \cong \xi G(i_{B})(G(B)) \subseteq M$ and since $M$ is  weakly uniserial, we assume that $h: \xi G(i_{A})(G(A)) \rightarrow  \xi G(i_{B})(G(B))$ is a monomorphism. Again, by \cite[Proposition 21.2]{AF}, we have the following monomorphism from $A$ to $B$:
\begin{center}
$ A \overset{\eta^{-1}} \longrightarrow FG(A) \cong F(\xi G(i_{A})(G(A))) \overset{F(h)} \longrightarrow F(\xi G(i_{B})(G(B))) \cong F(G(B)) \overset{\eta} \longrightarrow B $.
\end{center}
\end{proof}

A right $R$-module $M$ is called \textit{prime} if $M \neq 0$ and ${\rm Ann}_{R}(M) = {\rm Ann}_{R}(N)$, for any $0 \neq N \subseteq M$. A two sided ideal $P$ of $R$ is called an \textit{associated prime} of $M$ if there exists $N \subseteq M$ such that $P = {\rm Ann}_{R}(N)$ and $N$ is a prime $R$-module. The set of associated primes of $M$ is denoted by ${\rm Ass}(M)$. Also $M$ is called \textit{semiprime} (\textit{weakly prime}) if ${\rm Ann}_{R}(N)$ is a semiprime (prime) ideal of $R$, for any nonzero submodule $N$ of $M$. For more details see \cite{Karimi}.

\begin{Pro}\label{Ass(M)}
For a ring $R$, the following statements hold:\\
\indent {\rm (a)} If $M$ is a weakly uniserial right $R$-module, then $|{\rm Ass}(M)|\leq1$. Moreover, if $R$ satisfies \indent\indent in the ascending chain condition on two sided ideals, then $|{\rm Ass}(M)|=1$.\\
\indent {\rm (b)} If $M$ is a weakly uniserial right $R$-module and $N \subseteq M$ is a prime $R$-module, then  \indent\indent ${\rm Ann}_{R}(N)$ is a  maximal member in the set of $\left\lbrace{\rm Ann}_{R}(K) \mid  0\neq K\subseteq M \right\rbrace $.\\
\indent {\rm (c)} A weakly uniserial right $R$-module $M$ is weakly prime if and only if it is semiprime.\\
\indent {\rm (d)} A right {\rm(}left{\rm)} weakly uniserial ring $R$ is prime if and only if it is semiprime.
\end{Pro}

\begin{proof}
 (a). Assume that $P$ and $Q$ are distinct associated primes of $M$. Then $P= {\rm Ann}_{R}(N)$ and $Q={\rm Ann}_{R}(K)$, for some prime submodules $N$ and $K$ of $M$. Since $M$ is weakly uniserial we can assume that $ N \overset{f}\rightarrowtail K$ and so that $N\cong f(N) \subseteq K$. Since $K$ is prime we conclude that $P={\rm Ann}_{R}(N)={\rm Ann}_{R}(f(N))={\rm Ann}_{R}(K)=Q$, a contradiction. The second part of this statement follows from the first part and \cite[Lemma 3.58]{Lam}.

 (b). Assume that $0\neq K\subseteq M $ and ${\rm Ann}_{R}(N) \subseteq {\rm Ann}_{R}(K)$. Since $M$ is weakly uniserial, $ N \rightarrowtail K$ or $K \rightarrowtail N$. If $ N \overset{f} \rightarrowtail K$, then $N \cong f(N) \subseteq K$ and so ${\rm Ann}_{R}(K) \subseteq{\rm Ann}_{R}(f(N))={\rm Ann}_{R}(N)$. Therefore, in this case 
 ${\rm Ann}_{R}(N) = {\rm Ann}_{R}(K)$. If $ K \overset{g} \rightarrowtail N$, then $K \cong g(K) \subseteq N$ and since $N$ is prime we conclude that ${\rm Ann}_{R}(K)={\rm Ann}_{R}(g(K))={\rm Ann}_{R}(N)$, as desired.

 (c). Assume that $M$ is semiprime and there exists $N \subseteq M$ such that $NIJ=0$, for some ideals $I$ and $J$  of $R$.  Note that $N(JI)^{2} = NJIJI \subseteq NIJI=0$ and since $M$ is semiprime we conclude that $NJI=0$. By symmetry, let  $NI \rightarrowtail NJ$.  Then  ${\rm Ann}_{R}(NJ)\subseteq   {\rm Ann}_{R}(NI)$ and so  $NI^2=0$. Since  $M$ is semiprime,  $NI=0$,  as desired.  The converse is clear.

 (d). By substituting $R$ for $N$ in the proof of $\rm {(c)}$, the result is obtained
\end{proof}

\begin{Examp}
{\rm Since $\mathbb{Z} \times \mathbb{Z}$ is not a prime ring, by \Cref{Ass(M)}(d), it is not a right (left) weakly uniserial ring.}
\end{Examp}

\begin{Cor}
If $R$ is a semiprime right weakly uniserial ring, then ${\rm Soc}(R_{R}) = 0$ or $\Z(R_{R})=0={\rm J}(R_{R})$.   
\end{Cor}

\begin{proof}
We may assume that ${\rm Soc}(R_{R}) \neq 0$. By \cite[Proposition 7.13]{Lam}, ${\rm Soc}(R_{R}) \cap \Z(R_{R})=0$. Thus ${\rm Soc}(R_{R}) \Z(R_{R})=0$ and so by \Cref{Ass(M)}(d),  $\Z(R_{R})=0$.  On the other hand ${\rm Soc}(R_{R}) {\rm J}(R_{R})=0$ and again by \Cref{Ass(M)}(d), ${\rm J}(R_{R})=0$.   
\end{proof}

Recall that an element $a \in R$ is {\it right regular} if ${\rm r.Ann}_{R}(a)=0$. {\it Left regular} elements are defined similarly. As usual, $a \in R$ is {\it regular} if it is both a left and a right regular element. A ring $R$ is said to be {\it right hereditary} if every right ideal of $R$ is projective as a right $R$-module. Also, a ring $R$ is said to be {\it local} if $R$ has a unique maximal right ideal. The following proposition introduces classes and examples of weakly uniserial rings and modules. First we need the following lemma.
\begin{Lem}\label{free}
If $F_{1}$ and $F_{2}$ are free right $R$-modules, then $F_{1} \rightarrowtail F_{2}$ or $F_{2} \rightarrowtail F_{1}$.
\end{Lem}

\begin{proof}
The proof is routine.
\end{proof}

\begin{Pro}\label{domain}
For a ring $R$, the following statements hold:\\
\indent {\rm (a)} If every  nonzero right ideal of $R$ contains a right regular element, then $R$ is right weakly \indent\indent uniserial.\\
\indent {\rm (b)} Every domain is a right and  left weakly uniserial ring.\\ 
\indent {\rm (c)} If $R$ is either a principal right ideal domain or a right hereditary local ring, then every
\indent\indent  projective right $R$-module is weakly uniserial.
\end{Pro}

\begin{proof}
 (a). Assume that $I$ and $J$ are two nonzero right ideals of $R$ and $b$ is a right regular element of $J$. Then the map $f:I \rightarrow J$ by $f(a)=ba$ is an $R$-monomorphism and so $R$ is a right weakly uniserial ring.

 (b) follows from (a).

 (c). Let $R$ be a principal right ideal domain and $P$ be a projective right $R$-module. There exist a free right $R$-module $F$ and a right $R$-module $K$ such that $F \cong P \oplus K$. By \cite[Corollary 2.27]{Lam}, $P$ is a free right $R$-module. Then it is sufficient to prove that free modules over principal right ideal domains are weakly uniserial. Let $F$ be a free right $R$-module and $F_{1}$, $F_{2}$ are two nonzero submodules of $F$. By \cite[Corollary 2.27]{Lam}, $F_{1}$ and $F_{2}$ are free right $R$-modules and by \Cref{free}, $F_{1} \rightarrowtail F_{2}$ or $F_{2} \rightarrowtail F_{1}$. Therefore, $F$ is weakly uniserial.
 
Let $R$ be a right hereditary local ring and $P$ be a projective right $R$-module. Assume that $P_{1}$ and $P_{2}$ are nonzero submodules of $P$. Since $R$ is right hereditary, by \cite[Corollary 2.26]{Lam}, $P_{1}$ and $P_{2}$ are projective modules and since $R$ is local, $P_{1}$ and $P_{2}$ are free right $R$-modules. Now, by \Cref{free}, $P_{1} \rightarrowtail P_{2}$ or $P_{2} \rightarrowtail P_{1}$ and so  $P$ is a weakly uniserial right $R$-module.
\end{proof}

\begin{Rem}
\indent {\rm (a)} {\rm \Cref{domain}(b) shows that the class of weakly uniserial rings is a straightforward  generalization of domains.}
 
\indent {\rm (b)} 
{\rm The right hereditary and local conditions are necessary in \Cref{domain}(c). For example,
$\mathbb{Z}_{4}$ is local but is not a hereditary ring and  consider $ M =\mathbb{Z}_{4} \oplus \mathbb{Z}_{4} $ as a $\mathbb{Z}_{4}$-module. We show that $M$ is not weakly uniserial.  $N:=(\overline{1}, \overline{1}) \mathbb{Z}_{4} $ and $K:= 2\mathbb{Z}_{4} \oplus 2\mathbb{Z}_{4}$ are $\mathbb{Z}_{4}$-submodules of $M$ that $|N|=|K|$. Since
${\rm Ann}_{\mathbb{Z}_{4}}(N) \neq {\rm Ann}_{\mathbb{Z}_{4}}(K)$, we conclude that $N \not\rightarrowtail  K$ and $K \not\rightarrowtail N $. Therefore $M$ is not weakly uniserial. Also, $\mathbb{Z}_{6}$ is a hereditary (semisimple) and non-local ring. $\mathbb{Z}_{6}$ as a $\mathbb{Z}_{6}$-module is projective, but is not weakly uniserial.}
\end{Rem}

Let $R$ be a commutative principal ideal domain and $a,b \in R$. 
Then $a$ is said to {\it divide}  $b$ in $R$, denoted by $a|b$, if $b=ac$ for some $c\in R$. Two elements  $a$ and $b$ are {\it associates} if $b=au$, where $u$ is a unit of $R$. Also an element $p$ of $R$ is called {\it  irreducible} ({\it prime}) if it is neither zero nor a unit and if its only divisors are units and associates of $p$, i.e., for any $a,b\in R$, if $p|ab$, then we conclude that $p|a$ or $p|b$. Elements $a$ and  $b$ of $R$  are called {\it relatively prime} if $1$ is a  greatest common divisor of $a$ and $b$. In this case we write  ${\rm gcd}(a,b)=1.$

\begin{Examp}\label{examp domain}
{\rm (1) For any index set $I$ with $|I| \geq 2$, the $\mathbb{Z}$-module $\oplus_{I} \mathbb{Z}$ is neither uniserial nor almost uniserial but since $\mathbb{Z}$ is a principal ideal domain by the first part of \Cref{domain}(c), it is weakly uniserial.

(2) Let $\mathbb{Z}_{(p)} = \left\lbrace  a \slash b \in\mathbb{Q} ~ | ~ {\rm gcd}(a,b)={\rm gcd}(b,p)=1 \right\rbrace $ be the localization of $\mathbb{Z}$ at  the prime ideal  $p\mathbb{Z}$.  It is a hereditary (principal ideal domain) local ring. So  by the second part (as well as the first part) of \Cref{domain}(c), for any index set $I$ the $\mathbb{Z}_{(p)}$-module $\oplus_{I} \mathbb{Z}_{(p)}$ is weakly uniserial.

(3) For a nontrivial automorphism $\sigma : F \rightarrow F$ of a field $F$, let $R=F[[x; \sigma]]$ be a formal skew power series ring. Then $R$ is a right (and left) hereditary (principal right ideal domain) local ring and so by the second part (and also the first part) of \Cref{domain}(c), for any index set $I$ the right $R$-module $\oplus_{I} R$ is right weakly uniserial.}
\end{Examp}

In the following, we show that if $R$ is a right weakly uniserial ring, then a localization of $R$ is not necessarily a weakly uniserial right $R$-module.

\begin{Pro}\label{field of fractions}
Let $R$ be a commutative principal ideal domain and $Q$ be the field of fractions of $R$. Then the following statements hold:\\ 
\indent {\rm (a)} If $R$ is local, then $Q$ is a {\rm(}weakly{\rm)} uniserial $R$-module.\\
\indent {\rm (b)} If $R$ is not local, then $Q$ is not a weakly uniserial $R$-module. 
\end{Pro}

\begin{proof}
(a). Suppose that $R$ is local with the maximal ideal $\M$. Then $\M=pR$ where $p$ is the unique irreducible element of $R$ (up to multiplication by units). Let $I$ be a nonzero proper $R$-submodule of $Q$. If $I \subseteq R$, then $I=p^{n}R$ for some $n \in \mathbb{N} \cup \lbrace 0 \rbrace$. Now if $I \nsubseteq R$, then there exists $m \in \mathbb{N}$ such that $1\slash p^{m} \notin I$, because $I \neq Q $. We chose the smallest number $n$ such that $1\slash p^{n} \notin I$. Then $1\slash p^{n-1} \in I$ and so it is easy to see that $I=(1\slash p^{n-1}) R$. Thus the all of nonzero proper $R$-submodules of $Q$ are

\begin{center}
$\cdots \subseteq p^{2}R \subseteq pR \subseteq R \subseteq (1 \slash p) R \subseteq (1\slash p^{2}) R \subseteq \cdots$.
\end{center}

\begin{flushleft}
Therefore $Q$ is uniserial as an $R$-module.
\end{flushleft}
(b). Suppose that $R$ is not local and $p$, $q$ are non-associate  irreducible  elements of $R$. We claim that ${\rm Hom}_{R}(R_{(p)} , R_{(q)}) =0$ where
 $$ R_{(p)} = \{\ a \slash b \in Q ~ | ~{\rm gcd}(a,b)={\rm gcd}(b,p)=1 \},\ $$
 $$R_{(q)} = \{\ a \slash b \in Q ~| ~ {\rm gcd}(a,b)={\rm gcd}(b,q)=1 \}.\ $$ 
\\
Let $f \in {\rm Hom}_{R}(R_{(p)} , R_{(q)})$ and $f(a \slash b)= c \slash d \neq 0$, for some $a \slash b \in R_{(p)}$ and $c \slash d \in R_{(q)}$. Since every commutative principal ideal domain is a unique factorization domain, we can choose the positive integer $n$ such that $q^{n} \nmid c$. Now let $f(a \slash (bq^{n})) = x \slash y$. Then
$ c \slash d = f (a \slash b) = f ((aq^{n}) \slash (bq^{n})) = f(a \slash (bq^{n}))q^{n}$ thus $c \slash d = (x \slash y)q^{n}$. Hence $cy=dxq^{n}$ and so $c \slash (dq^{n}) = x \slash y $. Therefore $f(a \slash q^{n}b)=c \slash (q^{n}d) \notin R_{(q)}$, a contradiction.
\end{proof}

Let $R$ be a ring and $S$ be the set of all regular elements in a ring $R$. The ring $R$ is called {\it right Ore} if $S$ is right permutable, i.e., $aS \cap sR \neq \emptyset$, for any $a \in R $ and $s \in S$. \textit{Left Ore} ring is defined similarly. As usual, $R$ is called \textit{Ore} if it is both left and right Ore.
It is well-known that $R$ is right Ore if and only if  for any right $R$-module $M$, $t(M)= \lbrace m \in M \mid ms=0$ for some regular element $s \in R\rbrace$ is an $R$-submodule of $M$ (see \cite[Exercise 10.19]{Lam}). Let $M$ be a right $R$-module over a right Ore ring $R$. Then $M$ is called \textit{torsion} if $t(M)=M$ and \textit{torsion-free} if $t(M)=0$. In case $R$ is right Ore, ${M \slash t(M)}$ is always torsion-free.

\begin{Cor}\label{Q not weak}
The following statements hold:\\
\indent {\rm (a)} $\mathbb{Q}$ is not weakly uniserial as a $\mathbb{Z}$-module.\\
\indent {\rm (b)} If $M$ is not a torsion $\mathbb{Z}$-module, then $M_{(0)}$ {\rm (}the localization of $M$ at the prime ideal $(0)${\rm )} 
\indent\indent is not weakly uniserial as a $\mathbb{Z}$-module.
\end{Cor}

\begin{proof}
(a) follows from \Cref{field of fractions}(b).\\
\indent (b). Suppose that $x \in M$ and $nx \neq 0$, for any $0 \neq n \in \mathbb{Z}$. Then $\mathbb{Z} \cong x \mathbb{Z} \subseteq M$ and so $\mathbb{Q} = \mathbb{Z}_{(0)} \cong_{\mathbb{Q}} (x \mathbb{Z})_{(0)} \subseteq M_{(0)}$. Thus, $\mathbb{Q} = \mathbb{Z}_{(0)} \cong (x \mathbb{Z})_{(0)}$ as a $\mathbb{Z}$-module. Since by (a), $\mathbb{Q}_{\mathbb{Z}}$ is not weakly uniserial, $ (x \mathbb{Z})_{(0)}$ is not weakly uniserial. This implies that $M_{(0)}$ is not weakly uniserial as a $\mathbb{Z}$-module.
\end{proof}

In Section 4, a different proof for  \Cref{Q not weak}(a) is given.
Recall that a semisimple right $R$-module $M$ is said to be {\it homogeneous} if every two simple submodules of $M$ are isomorphic.

\begin{Rem}\label{RE}
{\rm Every uniserial module is uniform, however the class of uniform modules and the class of weakly uniserial modules are independent. For instance every not simple homogeneous semisimple module is  weakly uniserial, but it is not uniform. Also  $\mathbb{Q}$  is a uniform $\mathbb{Z}$-module that is not weakly uniserial by  \Cref{Q not weak}(a).}

\end{Rem}

\begin{Examp}\label{Some exam}
\rm{(1)} Every homogeneous semisimple right $R$-module is weakly uniserial.

\rm{(2)} The ring $ R= \left[
         \begin{array}{rr}
              \mathbb{Z}  & \frac{\mathbb{Z}}{2 \mathbb{Z} } \\
              0 &  \mathbb{Z}
          \end{array} \right]$ is neither a right nor a left weakly uniserial ring. To see this, consider two ideals  $I := \left[
         \begin{array}{rr}
             2 \mathbb{Z}  &  0 \\
              0 &  0
          \end{array} \right]$ and $ J := \left[
         \begin{array}{rr}
             0  & 0 \\
              0 &  2\mathbb{Z}
          \end{array} \right]$ of $R$. Since ${\rm r.Ann}_{R}(I) \nsubseteq {\rm r.Ann}_{R}(J)$ and ${\rm r.Ann}_{R}(J) \nsubseteq {\rm r.Ann}_{R}(I)$, we conclude that $ J \not\rightarrowtail I$ and $I \not\rightarrowtail J$. Also, note that ${\rm l.Ann}_{R}(I) \nsubseteq {\rm l.Ann}_{R}(J)$ and ${\rm l.Ann}_{R}(J) \nsubseteq {\rm l.Ann}_{R}(I)$.

\rm{(3)} Let $D$ be a division ring and
 $ R=\lbrace \left[
         \begin{array}{rrr}
              a  &  0 & b \\
              0  &  a & c \\
              0  &  0 & a
          \end{array} \right] | \ a, b, c \in D \rbrace$. Then, $R$ is both left and right weakly uniserial. Note that the only proper right (left) ideals of $R$ are
          $ \left[
         \begin{array}{rrr}
              0  &  0 & D \\
              0  &  0 & D \\
              0  &  0 & 0
          \end{array}\right]$,
          $ \left[
         \begin{array}{rrr}
              0  &  0 & D \\
              0  &  0 & 0 \\
              0  &  0 & 0
          \end{array} \right]$ and
     $ \left[
         \begin{array}{rrr}
              0  &  0 & 0 \\
              0  &  0 & D \\
              0  &  0 & 0
          \end{array} \right]$.
\end{Examp}

\begin{Lem}\label{homogeneous semisimple}
A semisimple right $R$-module is weakly uniserial if and only if it is homogeneous semisimple.
\end{Lem}

\begin{proof}
$(\Rightarrow).$  Set $M= \oplus_{i \in I} S_{i}$, where $S_{i}$ is a simple right $R$-module for any $i \in I$. Since $M$ is weakly uniserial, for any $i, j \in I$, $S_{i} \rightarrowtail S_{j}$ or $S_{j} \rightarrowtail S_{i}$. In any case, we conclude that $S_{i} \cong S_{j}$, as required. The converse is clear.
\end{proof}

Recall that a ring $R$ is said to be {\it right} (\textit{left}) \textit{duo} if every right (left) ideal of $R$ is two-sided. Also a ring $R$ is called {\it right} (\textit{left}) \textit{semi-duo} if every maximal right (left) ideal of $R$ is two-sided.

\begin{Cor}\label{semi-duo}
A right semi-duo ring $R$ is local if and only if every semisimple right $R$-module is weakly uniserial.
\end{Cor}

\begin{proof}
Assume that every semisimple right $R$-module is weakly uniserial and ${\mathcal M_1}$ and ${\mathcal M_2}$  are two distinct maximal right ideals of $R$. In right $R$-module  $ R \slash {\mathcal M_1} \oplus R \slash {\mathcal M_2}$,  by \Cref{homogeneous semisimple}, $R \slash {\mathcal M_1} \cong R \slash {\mathcal M_2}$. But since $R$ is right semi-duo,  ${\mathcal M_1}={\rm Ann}_{R}(R \slash {\mathcal M_1})={\rm Ann}_{R}(R \slash {\mathcal M_2})={\mathcal M_2}$ and hence  $R$ is local. The converse is clear by \Cref{homogeneous semisimple}.
\end{proof}

A right $R$-module $M$ is called \textit{semi-Artinian} if for every submodule $N \neq M$, ${\rm Soc}(M/N) \neq 0$.
A ring $R$ is a \textit{right semi-Artinian ring}, if  $R$ is a semi-Artinian right $R$-module. A ring $R$ is right semi-Artinian if and only if every nonzero right $R$-module has an essential socle, if and only if every nonzero right $R$-module has a simple submodule.

\begin{Pro}\label{essential socle}
For a ring $R$, the following statements hold:\\
\indent {\rm (a)} If $M$ is a weakly uniserial right $R$-module with nonzero socle, then ${\rm Soc}(M)$ is a homoge-
\indent\indent neous semisimple essential submodule of $M$.\\
\indent {\rm (b)} If $R$ is a right semi-Artinian, then the socle of any weakly uniserial right $R$-module $M$ is \indent \indent  a  homogeneous semisimple essential submodule of $M$.
\end{Pro}

\begin{proof}
\rm{(a)}. By \Cref{submodule}(1) and \Cref{homogeneous semisimple}, ${\rm Soc}(M)$ is homogeneous semisimple. Now, suppose that $S$ is a simple submodule of $M$. Then by hypothesis, for every nonzero submodule $N$ of $M$, $N \rightarrowtail S$ or $S \rightarrowtail N$. In any case, $N$ contains a simple submodule and so ${\rm Soc}(M)\subseteq_{e} M$.

\rm{(b)} follows from  (a).
\end{proof}

The \textit{uniform dimension} of a right $R$-module $M$ {\rm (}denoted by ${\rm u.dim}(M){\rm )}$ is the supremum of the set  $\{k ~| ~M ~{\rm contains ~a ~direct ~sum ~of} ~k ~ {\rm nonzero ~submodules} \}$. Also $M$ is called \textit{Dedekind-finite} if for every right $R$-module $N$,  the relation $M \cong M \oplus N$ implies that $N = 0$. It is easy to check that any right $R$-module with finite uniform dimension is Dedekind-finite.

\begin{The}\label{N and Socle}
For a ring $R$, the following statements hold:\\
\indent {\rm (a)} If $M$ is a weakly uniserial right $R$-module such that ${\rm Soc}(M)$ is Dedekind-finite, then \indent\indent $N \subseteq {\rm Soc}(M)$ or ${\rm Soc}(M) \subseteq N$, for any submodule $N$ of $M$. Moreover, ${\rm Soc}(M)=0$ or $M$ \indent\indent is indecomposable or $M$ is homogeneous semisimple.\\
\indent {\rm (b)} If $R$ is a right weakly uniserial ring such that ${\rm Soc}(R_{R})$ is Dedekind-finite, then ${\rm J}(R)^{2}=0$ 
\indent\indent or ${\rm Soc}^{2}(R_{R})= 0$. Moreover, ${\rm Soc}(R_{R})=0$ or $R_{R}$ is indecomposable or $R \cong M_{n}(D)$ \indent\indent where $D$ is a division ring.
\end{The}

\begin{proof}
 (a). Assume that $N$ is a nonzero submodule of $M$. Since $M$ is weakly uniserial, $N  \rightarrowtail {\rm Soc}(M)$ or ${\rm Soc}(M) \rightarrowtail N$. If $N   \rightarrowtail {\rm Soc}(M)$, then $N$ is semisimple and so   $N \subseteq {\rm Soc}(M)$. If ${\rm Soc}(M) \overset{f} \rightarrowtail N$, then ${\rm Soc}(M) \cong f({\rm Soc}(M)) \subseteq N$  and so there exists a submodule $K$ of $M$ such that $f({\rm Soc}(M)) \oplus K = {\rm Soc}(M)$. Thus ${\rm Soc}(M) \oplus K \cong {\rm Soc}(M)$ and since ${\rm Soc}(M)$ is Dedekind-finite we conclude that $K=0$. Therefore, ${\rm Soc}(M) = f({\rm Soc}(M)) \subseteq N$.

For the second part of this statement, assume that $M$ is not indecomposable. Then there exist two nonzero submodules $M_{1}$ and $M_{2}$ of $M$ such that $M=M_{1} \oplus M_{2}$. If ${\rm Soc}(M) =0$, we are done but if ${\rm Soc}(M) \neq 0$, then by \Cref{essential socle}(a), ${\rm Soc}(M) \subseteq_{e} M$ and so by the first part of the proof we conclude that $M_{i} \subseteq {\rm Soc}(M)$, for $i =1, 2$. Then $M=M_{1} \oplus M_{2} \subseteq {\rm Soc}(M)$ and so by \Cref{homogeneous semisimple}, $M$ is homogeneous semisimple.

 (b). By (a), ${\rm Soc}(R_{R})\subseteq {\rm J}(R)$ or ${\rm J}(R)\subseteq {\rm Soc}(R_{R})$. On the other hand, ${\rm Soc}(R_{R}){\rm J}(R)=0$ and so ${\rm J}(R) \subseteq {\rm Soc}(R_{R})$ implies that ${\rm J}(R)^{2}=0$ and ${\rm Soc}(R_{R})\subseteq {\rm J}(R)$ implies that ${\rm Soc}^{2}(R_{R})=0$. The second part follows from (a).
\end{proof}

Artinian principal ideal rings were studied in papers of G. K{\"o}the and K. Asano, where it was proved that any Artinian principal right ideal ring is right uniserial (see \cite{AS1,AS2}). In fact, K. Asano proved that an Artinian ring is uniserial if and only if each ideal is a principal right ideal and a principal left ideal. In the following result we give the structure of some right weakly uniserial rings such as the right Artinian principal right ideal rings. A ring $R$ is called \textit{left perfect} if $R \slash {\rm J}(R)$ is semisimple and ${\rm J}(R)$ is left T-nilpotent. 

\begin{Cor}\label{right artinian right weakly}
For a ring $R$, the following statements hold:\\
\indent {\rm (a)} Let $R$ be a left perfect ring such that ${\rm Soc}(R_{R})$ is Dedekind-finite. If $R$ is right weakly \indent\indent uniserial, then $R$ is local or $R \cong M_{n}(D)$ where $D$ is a division ring.

\indent {\rm (b)} If $R$ is a right Artinian right weakly uniserial ring, then $R$ is local or $R \cong M_{n}(D)$ where \indent\indent $D$ is a division ring.

\indent {\rm (c)} If $R$ is a right Artinian principal right ideal ring, then $R$ is right weakly uniserial if and \indent\indent only if $R$ is right uniserial or $R \cong M_{n}(D)$ where $D$ is a division ring.
\end{Cor}

\begin{proof}
 (a). Since $R$ is left perfect, by \cite[Theorem 23.20]{Lam1}, ${\rm Soc}(R_{R}) \neq 0$. So by \Cref{N and Socle}(b), if $R\ncong M_{n}(D)$, then $R_{R}$ is indecomposable and so $R$ has no non-trivial idempotents. Therefore by \cite[Corollary 21.29]{Lam1}, $R$ is a local ring.

 (b) follows from (a).

 (c). Assume that $R$ is right weakly uniserial. Since $R$ is right Artinian, ${\rm u.dim}({\rm Soc}(R_{R})) < \infty$ and so ${\rm Soc}(R_{R})$ is Dedekind-finite. Hence by (a) if $R \ncong M_{n}(D)$, then $R$ is a local ring and so by \cite[Lemma 2.13]{BehQ}, $R$ is a right uniserial ring. Conversely,  suppose that $R \cong M_{n}(D)$ where $D$ is a division ring. Then $R_{R}$ is homogeneous semisimple and so by \Cref{Some exam}(1), $R_{R}$ is weakly uniserial. Also clearly right uniserial rings are right weakly uniserial.
\end{proof}

\begin{Examp}
{\rm Let $D$ be a division ring and 
 $R= \left[
         \begin{array}{rr}
              D  & D \\
              0 &  D 
          \end{array} \right]$. Clearly $R$ is a right (and left) Artinian ring. Since 
           $ \left[
          \begin{array}{rr}
              D & D \\
              0 &  0 
          \end{array} \right]$ and 
          $ \left[
          \begin{array}{rr}
              0 & D\\
              0 &  D 
          \end{array} \right]$ are two maximal right (and left) ideals of $R$, then $R$ is not a local ring. Also ${\rm J}(M_{n}(D)) =0$ and ${\rm J}(R)= \left[
          \begin{array}{rr}
              0  & D \\
              0 &  0 
          \end{array} \right] \neq 0$, hence $R \ncong M_{n}(D)$. Therefore by \Cref{right artinian right weakly}(b), $R$ is neither a right nor a left weakly uniserial ring.}
\end{Examp}

A ring $R$ is called \textit{right Kasch} if every simple right $R$-module can be emmbedded in $R_{R}$.  \textit{Left Kasch} ring is defined similarly. As usual, $R$ is called a \textit{Kasch} ring if it is both right and left Kasch.

\begin{Pro}\label{comm. Kasch}
For a ring $R$, the following statements hold:\\
\indent {\rm (a)} If $R$ is a commutative $\left( semi\right)$ prime weakly uniserial ring, then $R$ is a field or ${\rm Soc}(R)=0$. \\
\indent {\rm (b)} Every commutative Kasch weakly uniserial ring is local.\\
\indent {\rm (c)} Every commutative Artinian weakly uniserial ring is local.
\end{Pro}

\begin{proof}
 (a). Assume that ${\rm Soc}(R) \neq 0$. Then there exists a nonzero minimal ideal $I$ of $R$. By \cite[Brauer's Lemma 10.22]{Lam1}, $I^{2}= 0$ or $I=eR$, for some idempotent $e \in R$. If $I^{2}= 0$, then since $R$ is semiprime we conclude that $I=0$, a contradiction. Hence $I=eR$, but since $R$ is weakly uniserial,  we have $e=0$ or $e=1$. If $e=0$, then $I=0$, a contradiction. Thus $e=1$, and so ${\rm Soc}(R)=R$. Now by \Cref{homogeneous semisimple}, $R_{R}$ is homogeneous semisimple and so $R \cong M_{n}(D)$ where $D$ is a division ring. Now, since $R$ is commutative we conclude that $R$ is a field.

 (b). Suppose that $R$ is a commutative weakly uniserial Kasch ring and $\M_{1}, \M_{2}$ are maximal ideals of it. Since $R$ is Kasch, $R/ \M_{1}$ and $R/ \M_{2}$ are embedded in $R$ and since $R$ is weakly uniserial,  $R/ \M_{1}\rightarrowtail R/ \M_{2}$ or $R/ \M_{2}\rightarrowtail R/ \M_{1}$. In both cases we have $R/ \M_{1} \cong R/ \M_{2}$ and so  $\M_{1} = {\rm Ann}_{R}(R/ \M_{1}) = {\rm Ann}_{R}(R/ \M_{2}) =\M_{2} $, as desired.

 (c) follows from \Cref{right artinian right weakly}(b).
\end{proof}

Let $N$ be a submodule of a right $R$-module $M$. A submodule $C \subseteq M$ is said to be a complement to $N$ (in $M$) if $C$ is maximal with respect to the property that $C \cap N = 0$. A submodule $C$ of a right $R$-module $M$ is called a complement (in $M$), if there exists a submodule $N$ of $M$ such that $C$ is a complement to $N$ (in $M$). 
A right $R$-module $M$ is called \textit{CS}, if every complement in $M$ is a direct summand of $M$. Also $M$ is called \textit{cocyclic} if it has an essential simple submodule. Note that CS modules are not necessarily weakly uniserial. For instance  $\mathbb{Q}_{\mathbb{Z}}$  is a CS module that is not weakly uniserial (see \Cref{Q not weak}(a)).

\begin{Pro}\label{CS}
If $M$ is a weakly uniserial CS right $R$-module such that ${\rm Soc}(M)$ is Dedekind-finite, then ${\rm Soc}(M)=0$ or $M$ is cocyclic or $M$ is homogeneous semisimple.
\end{Pro}

\begin{proof}
Assume that ${\rm Soc}(M) \neq 0$. If ${\rm Soc}(M)=S$ where $S$ is a simple submodule of $M$, then by \Cref{essential socle}(a), $S\subseteq_{e} M$ and so $M$ is cocyclic. Now suppose that ${\rm Soc}(M)=\oplus_{I} S_{i}$, where $|I| \geq 2$ and every $S_{i}$ is a simple submodule of $M$. First, we show that every $S_{i}$ is a complement in $M$. Let $t \in I $ and $N$ be a submodule of $M$ such that $S_{t} \subseteq N$ and $ N \cap (\oplus_{t \neq i \in I} S_{i})= 0 $. Since ${\rm Soc}(M) \nsubseteq N $, by \Cref{N and Socle}(a), we conclude that $N \subseteq {\rm Soc}(M)$. Suppose that $x \in N$. Then $ x = x_t+ \sum_{j=1}^{n} x_{{i}_j}$, where $x_{{i}_j} \in S_{{i}_j}$, $x_t\in S_t$ and $ {i}_j \in I\setminus \{t\}$, for $j=1, \ldots, n$. Thus $x-x_{t}= \sum_{j=1}^{n} x_{{i}_j}$ and so  $x=x_{t}$. It  follows that $N=S_{t}$ and hence $S_{t}$ is a complement in $M$. Since $M$ is a CS module, there exists a nonzero submodule $L$ of $M$ such that $S_{t} \oplus L = M$. So,  $M$ is decomposable and ${\rm Soc}(M) \neq 0$. Therefore, \Cref{N and Socle}(a) implies that $M$ is a homogeneous semisimple right $R$-module.
\end{proof}

\begin{Cor}
Every weakly uniserial uniform right $R$-module with nonzero socle is cocyclic.
\end{Cor}

\section{\hspace{-5mm}. Rings over which every module  is  weakly uniserial}

In this section, we characterize  rings over which every (fnitely generated or injective)
module is weakly uniserial. A ring is called  \textit{quasi-Frobenius} (brieﬂy, QF ) if it is right Noetherian and right self-injective. The following example shows that projective and injective modules are not necessarily weakly uniserial.

\begin{Examp} \label{Projective and injective}
{\rm Let $R=\mathbb{Z}_{p^{2}}$ and $M =R \oplus R $, where $p$ is a prime number. Since  $M$ is a projective $R$-module and  $R$ is QF,  $M$  is an injective $R$-module.  Now, $N:=(\overline{1}, \overline{1}) R$ and $K:= pR \oplus pR$ are two $R$-submodules of $M$ that $|N|=|K|$. But since ${\rm Ann}_{R}(N) \neq {\rm Ann}_{R}(K)$, we conclude that $N \not\rightarrowtail  K$ and $K \not\rightarrowtail N $. Therefore, $M$ is not a weakly uniserial $R$-module.}
\end{Examp}

In the following we show that cyclic modules are not necessarily weakly uniserial.

\begin{Examp} \label{Cyclic}
\rm{(1)} As stated in \Cref{central idempotent}(a), a ring with nontrivial central idempotents is not weakly uniserial. For instance consider the ring  $\mathbb{Z}/{6\mathbb{Z}}$. 

\rm{(2)}  {\rm Let $K$ be a field and $R$ be the $K$-algebra with generators $x$ and $y$ such that $x^{3}=y^{3}=xy=0$ (i.e., $R \cong K[x,y] \slash \langle x^{3}, y^{3}, xy \rangle$). $R$ is an Artinian local ring with the maximal ideal $\M=xR \oplus yR$. Also, ${\rm Spec}(R) =\left\lbrace  \M \right\rbrace $. If $ xR \rightarrowtail yR$, then  $ {\rm Ann}_{R}(yR) \subseteq {\rm Ann}_{R}(xR)$ and so  $x^{2}=0$, a contradiction. If $yR \rightarrowtail xR$, then similarly $y^{2}=0$, a contradiction. Therefore  neither $\M$  is  weakly uniserial, nor $R_R$.   }
\end{Examp}

 The following lemma plays a key role in the subsequent results.

\begin{Lem}\label{5 property}
For a ring $R$ that every $2$-generated right $R$-module is weakly uniserial, the following statements hold:\\
\indent {\rm (a)} All simple right $R$-modules are isomorphic.\\
\indent {\rm (b)} $R$ is a right semi-Artinian right Kasch ring.\\
\indent {\rm (c)} $R$ has only one prime ideal.\\
\indent {\rm (d)} All two-sided ideals are comparable.
\end{Lem}

\begin{proof}
 (a). Set $M= S_{1} \oplus S_{2}$ where $S_{1}$ and $S_{2}$ are simple right $R$-modules. Since $M$ is weakly uniserial, by \Cref{homogeneous semisimple}, $S_{1} \cong S_{2}$.

 (b). Set $M=(R \slash \M) \oplus C$ where $\M$ is a maximal right ideal of $R$ and $C$ is a cyclic right $R$-module. Since $M$ is weakly uniserial, $R \slash \M \rightarrowtail C$ or $C \rightarrowtail R \slash \M$. In any case ${\rm Soc}(C) \neq 0$ and so every nonzero right $R$-module has a simple submodule. Therefore, $R$ is a right semi-Artinian ring. Also, if we set $M=R \oplus S$ where $S$ is a simple right $R$-module, then since $M$ is weakly uniserial, $R \rightarrowtail S$ or $S \rightarrowtail R$. In any case, $S$ is embedded in $R$ and so $R$ is a right Kasch ring.

 (c). Assume that $P_{1}$ and $P_{2}$ are distinct prime ideals of $R$  and set $M= (R \slash P_{1}) \oplus (R \slash P_{2})$. Then ${\rm Ass}(M)={\rm Ass}(R \slash P_{1}) \cup {\rm Ass}(R \slash P_{2})  = \lbrace P_{1} ,  P_{2} \rbrace$. Thus by \Cref{Ass(M)}(a), $M$ is not weakly uniserial, a contradiction.

 (d). Set $M=R \slash I \oplus R \slash J$ where $I$ and $J$ are two-sided ideals of $R$. Since $M$ is weakly uniserial, $ R \slash I \rightarrowtail R \slash J $ or $ R \slash J \rightarrowtail R \slash I$. Therefore, $J={\rm r.Ann}_{R}(R \slash J) \subseteq {\rm r.Ann}_{R}(R \slash I)=I$ or $I={\rm r.Ann}_{R}(R \slash I) \subseteq {\rm r.Ann}_{R}(R \slash J)=J$.
\end{proof}

\begin{Lem}\label{soc}
 Let $R$ be a ring  that all simple right  $R$-modules are isomorphic. Then
 $R$ is homogeneous  semisimple or  ${\rm Soc}^{2}(R_{R}) =0$.
\end{Lem}

\begin{proof}
We may assume  ${\rm Soc}(R_{R}) \neq 0$. It is well-known that every  simple right  $R$-module is either singular or projective, but not both. By hypothesis,  all  simple right  $R$-modules  are either  singular or projective. If all  simple right  $R$-modules  are  projective, then it is easy to see that $R$ is semisimple and by hypothesis,  $R$ is homogeneous  semisimple (see \cite[p.63 Theorem 1.1]{Chen}).  Thus assume that all  simple right  $R$-modules  are  singular. Then  any simple right ideal $I$ of $R$ is singular and by  \cite[Lemma 7.2]{Lam}, $I ({\rm Soc}(R_{R}))=0$ and hence    ${\rm Soc}^{2}(R_{R}) =0$.  
 \end{proof}

Let $\mu$ be an ordinal and $\mathcal{A}=\left( A_{\alpha}~ | ~\alpha \leq \mu \right)$  be a sequence of modules. $\mathcal{A}$ is called \textit{continuous chain} provided that $A_{0}=0$, $A_{\alpha} \subseteq A_{\alpha +1}$  for all   $\alpha < \mu $ and $A_{\alpha}=\bigcup _{\beta < \alpha} A_{\beta}$ for all limit ordinals $\alpha \leq \mu$.
Let $M$ be a module and $\mathcal{C}$ be a class of modules. $M$ is \textit{$\mathcal{C}$–filtered}, if there are an ordinal $\kappa$ and a continuous chain of modules, $\left( M_{\alpha} ~| ~\alpha \leq \kappa \right) $, consisting of submodules of $M$ such that $M=M_{\kappa}$, and each of the modules $M_{\alpha +1} \slash M_{\alpha}~\left( \alpha < \kappa \right) $ is isomorphic to an element of $\mathcal{C}$. In this case the chain $ \left( M_{\alpha} ~| ~\alpha \leq \kappa \right) $ is called a \textit{$\mathcal{C}$–filtration} of $M$. If $R$ is a ring, $M$ is a semi-Artinian  right $R$-module and  $\mathcal{C}$ is the class of all simple right $R$-modules, then it is easy to see that $M$  is  $\mathcal{C}$–filtered.
Also, a ring $R$ is called \textit{right $V$-ring} if every simple right $R$-module is injective. The following theorem  answers  Question $1$.

\begin{The}\label{main theorem}
For a ring $R$, the following statements are equivalent:\\
\indent {\rm (a)} $R \cong M_{n}(D)$ where $D$ is a division ring.\\
\indent {\rm (b)} Every right $R$-module is weakly uniserial.\\
\indent {\rm (c)} Every finitely generated right $R$-module is weakly uniserial.\\
\indent {\rm (d)} Every $2$-generated right $R$-module is weakly uniserial.\\
\indent {\rm (e)} Every injective right $R$-module is weakly uniserial.\\
\indent {\rm (f)} The left-right symmetric of $(b)$, $(c)$, $(d)$ and $(e)$.
\end{The}

\begin{proof}
 (a) $\Rightarrow$ (b) $\Rightarrow$ (c) $\Rightarrow$ (d) are  clear.\\
 (d) $\Rightarrow$ (a). Assume that every  $2$-generated right $R$-module is weakly uniserial. By \Cref{5 property}(a), there is only one (up to isomorphic) simple right $R$-module, say $S$. We claim that every short exact sequence of the form 
$ 0 \longrightarrow S \overset{f} \longrightarrow X  \overset{g} \longrightarrow S \longrightarrow 0 $
splits where $X$ is a right  $R$-module. For if not, it is shown that $X$ is a cyclic module and $X \oplus S$ is not weakly uniserial, which  contradicts (d). First we show that $X$ is a  module of length  $2$. 
By exactness of  above short exact sequence, we have  $Imf$  both a maximal and minimal submodule of $X$.
 Then $0 \subsetneq Imf \subsetneq X$ is a composition series for $X$ and so  $l(X)=2$ (the length of $X$). \\
Suppose that $X_{1}$ and $X_{2}$ are nonzero proper submodules of $X$. We show that $X_{1} \cap X_{2} \neq 0 $. For if not,  $X_{1} \cap X_{2} = 0 $ and so $0 \subsetneq X_{1} \subsetneq X_{1} \oplus X_{2} \subseteq X$.  Since  $l(X)=2$, we have  $X= X_{1} \oplus X_{2} $ and  there exists a maximal submodule $T$ of $X$ such that $0 \subsetneq X_{1} \subseteq T \subsetneq X$. Hence $T= X_{1}$ and $X \slash X_{1}  \cong X_{2}$ implies that $X_{2}$ is a simple submodule of $X$ and so $X_{2} \cong S$. Similarly, $X_{1} \cong S$  and so $X= X_{1} \oplus X_{2} \cong S \oplus S$, a contradiction.
Thus $X_{1} \cap X_{2} \neq 0 $ and we have $0 \subsetneq X_{1} \cap X_{2} \subseteq X_{1}  \subsetneq X$. But since $l(X)=2$,  $X_{1} \cap X_{2} = X_{1}$ and so $ X_{1} \subseteq  X_{2}$. Similarly, $ X_{2} \subseteq  X_{1}$ and therefore $ X_{1} = X_{2}$. This clearly shows that X is cyclic.\\
Now to  show that $X \oplus S$ is not weakly uniserial, we consider two submodules $X$ and $N \oplus S$  of $X \oplus S$ where $N$ is a maximal (minimal) submodule of $X$ and  proceed by cases:\\
\textbf{ Case 1}:   $N \oplus S \overset{\varphi} \rightarrowtail X$. Then  $N \oplus S \cong \varphi(N \oplus S)= \varphi(N) \oplus \varphi(S) \subseteq X$. Hence $0 \subsetneq \varphi(S) \subsetneq \varphi(N) \oplus \varphi(S) \subseteq X$ and so we have $\varphi(N) \oplus \varphi(S) = X$, since $l(X)=2$. Therefore $X=\varphi(N) \oplus \varphi(S) \cong S \oplus S $, a contradiction.\\
\textbf{Case 2}:   $X \overset{\psi} \rightarrowtail N \oplus S$. Then $X \cong \psi(X) \subseteq N \oplus S $. But $l(N)=1$ and so $N \cong S$. Thus $X$ is a semisimple  module of length $2$, i.e., $X \cong S \oplus S$, a contradiction. \\
Therefore, every short exact sequence of the form 
 $0 \longrightarrow S \longrightarrow X  \longrightarrow S \longrightarrow 0$ splits. This implies that ${\rm Ext}_{R}^{1}(S,S)=0$. Now assume that $M$ is a right $R$-module. By  \Cref{5 property}(b), $M$ is  semi-Artinian  and so it has an $S$-filtration, that is, there exist an ordinal $\kappa$ and an increasing continuous chain of submodules $0=M_{0}\subseteq M_{1} \subseteq M_{2} \subseteq \cdots \subseteq M_{\kappa}=M$ such that 
 $M_{i+1} \slash M_{i} \cong S$ for every $i < \kappa$. Thus ${\rm Ext}_{R}^{1}(M_{i+1} \slash M_{i} ,S)=0$ and so, by Eklof’s Lemma \cite[Lemma 3.1.2]{gobel}, ${\rm Ext}_{R}^{1}(M,S)=0$. Hence, by \cite[Lemma 5.49]{Lam}, $S$ is an injective right $R$-module and so $R$ is a right $V$-ring. Now by  \cite[Theorem 6.2]{Jain}, ${\rm Soc}^{2}(R_{R})={\rm Soc}(R_{R})$.
 Since  ${\rm Soc}(R_{R})\neq 0$, by  \Cref{soc}, $R$ is homogeneous semisimple and hence $R \cong M_{n}(D)$ where $D$ is a division ring.
 
  (d) $\Leftrightarrow$ (e) is clear, since  every module is a submodule of its injective hull.
\end{proof}

\begin{Cor}\label{right duo $2$-generated}
For a right semi-duo ring $R$, the following statements are equivalent:\\
\indent {\rm (a)}  $R$ is a division ring.\\
\indent {\rm (b)}  Every finitely generated free right $R$-module is weakly uniserial and $ {\rm Soc}(R_{R})$ is a minimal \indent\indent   ideal of $R$.\\
\indent {\rm (c)}  $R \oplus R$ is a weakly uniserial right $R$-module and  $ {\rm Soc}(R_{R})$ is a minimal ideal of $R$.\\
\indent {\rm (d)}  Every {\rm (}$2$-generated{\rm )} right $R$-module is weakly uniserial. 
\end{Cor}

\begin{proof} 
 (a) $\Rightarrow$ (b) $\Rightarrow$ (c) and (a) $\Rightarrow$ (d)  are clear.

(c) $\Rightarrow$ (a). Let $K=(0) \oplus R$ and $N=I \oplus I$ be submodules of  $R \oplus R$ where $I$ is the minimal right ideal of $R$ (i.e., $I={\rm Soc}(R_{R})$). Since $R \oplus R$ is weakly uniserial, $N \rightarrowtail K$ or $K \rightarrowtail N$. In the first case, $I \oplus I \overset{f} \rightarrowtail R$ and so $f(I) \oplus f(I) = f(I \oplus I) =f({\rm Soc}(I \oplus I)) \subseteq {\rm Soc}(R_{R})$, a contradiction. Thus, $ R \rightarrowtail I \oplus I $ and so ${\rm r.Ann}_{R}(I)={\rm r.Ann}_{R}(I) \cap {\rm r.Ann}_{R}(I)= {\rm r.Ann}_{R}(I \oplus I) \subseteq {\rm r.Ann}_{R}(R)=0$. But $I \cong R \slash \M$ where $\M$ is a maximal right ideal of $R$ and since $R$ is  right semi-duo, $\M={\rm r.Ann}_{R}(R \slash \M)={\rm r.Ann}_{R}(I)= 0$. This implies that $R$ is a division ring.

(d) $\Rightarrow$ (a).  Assume that every {\rm (}$2$-generated{\rm )} right $R$-module is weakly uniserial. By  \Cref{main theorem}, $R \cong M_{n}(D)$ where $D$ is a division ring. On the other hand since $R$ is right semi-duo, similar to the proof of \Cref{semi-duo}, $R$ is local. Therefore, $R$ is a division ring.
\end{proof}

Recall that a ring $R$ is called right $p.p$-ring if each principal right ideal of $R$ is projective.

\begin{Pro}\label{main theorem2}
 For a ring $R$, the following statements are equivalent:\\
\indent{\rm (a)} $R \cong M_{n}(D)$ where $D$ is a division ring.\\
\indent{\rm (b)} $R$ is a right $p.p$-ring  such that  $R \oplus S$ is a weakly uniserial right $R$-module for any simple \indent\indent right $R$-module $S$.\\
\indent{\rm (c)}  $ {\rm Soc}(R_{R})$ is nonzero  Dedekind-finite and $R \oplus R$ is a weakly uniserial right  $R$-module.
\end{Pro}

\begin{proof}
(a) $\Rightarrow$ (b) is clear.

(b) $\Rightarrow$ (a).  Assume that $R$ is a right $p.p$-ring and  $ S$ is a  simple right  $R$-module. By hypothesis,  $R \rightarrowtail S$ or $S \rightarrowtail R$. In the first case $R$ is a simple right $R$-module and we are done. If  $S \overset{f} \rightarrowtail R $, then $S \cong f(S) \subseteq R$ and so $f(S)$ is a principal right ideal of $R$ and $S$ must be  projective. Now, it is easy to see that  $R$ is a semisimple ring and since it is  right weakly uniserial, \Cref{homogeneous semisimple} implies that $R \cong M_{n}(D)$ where $D$ is a division ring.
 
(a) $\Rightarrow$ (c). Follows from \Cref{main theorem}.

(c) $\Rightarrow$ (a). Let  $R$ be a ring such that $ {\rm Soc}(R_{R})$ is nonzero Dedekind-finite and set $M=R \oplus R$. We  consider two submodules $(0) \oplus R $ and $I \oplus I$ of $M$, where $I$ is a minimal right ideal of $R$. Since $M$ is weakly uniserial, $ (0) \oplus R \rightarrowtail I \oplus I $ or $ I \oplus I \rightarrowtail (0) \oplus R$. In the first case $R$ is semisimple. In the latter case $ I \oplus I  \overset{f} \rightarrowtail R$ and so $ I \oplus I \cong f(I \oplus I)= f(I) \oplus f(I) \subseteq R$. Since $f$ is monomorphism,  $f(I)$ is a minimal right ideal of $R$. Set $I_{1}=f(I)$ and again since $M$ is weakly uniserial, $ (0) \oplus R  \rightarrowtail (I_{1} \oplus I_{1}) \oplus I_{1}$ or $(I_{1} \oplus I_{1} )\oplus I_{1} \rightarrowtail (0) \oplus R$. In the first case $R$ is  semisimple. In the second case $ (I_{1} \oplus I_{1}) \oplus I_{1} \overset{f} \rightarrowtail R$ and so $(I_{1} \oplus I_{1}) \oplus I_{1} \cong f(I_{1} \oplus I_{1} \oplus I_{1}) =  f(I_{1}) \oplus f(I_{1}) \oplus f(I_{1}) \subseteq R$. If we set $I_{2} = f(I_{1})$,  then $R$ contains $I_{2} \oplus I_{2} \oplus I_{2}$, where $I_{2}$ is a minimal right ideal of $R$. By continuing this process if $R$ is not semisimple, then $\oplus_{i =1}^{\infty} I_{i} \subseteq {\rm Soc}(R_{R})$ where every $I_{i}$ is a minimal right ideal of $R$ and $I_{i} = I_{j}$ for every $ i, j \geq 1$. Since $\oplus_{i =1}^{\infty} I_{i} \cong \oplus_{i =2}^{\infty} I_{i} \oplus I_{1}$ and $\oplus_{i =1}^{\infty} I_{i} \cong \oplus_{i =2}^{\infty} I_{i}$, we conclude that $\oplus_{i =1}^{\infty} I_{i}$ is not Dedekind-finite. On the other hand it is easy to see that every direct summand of Dedekind-finite modules is also Dedekind-finite. Now since ${\rm Soc}(R_{R})$ is semisimple and $\oplus_{i =1}^{\infty} I_{i}$ is a direct summand of it, we conclude that $\oplus_{i =1}^{\infty} I_{i}$ is Dedekind-finite, a contradiction. Therefore,  $R$ is a semisimple ring and since $R$ is right weakly uniserial, by \Cref{homogeneous semisimple},  $R \cong M_{n}(D)$ where $D$ is a division ring.
\end{proof}

It is easy to see that every cyclic right $R$-module is uniserial if and only if $R$ is a right uniserial ring. But this is not necessarily true for weakly uniserial modules. For example $\mathbb{Z}$ is a weakly uniserial ring, but $\mathbb{Z}_{6}$ is a cyclic $\mathbb{Z}$-module that is not weakly uniserial. This fact and above results motivate the following question:

\begin{Ques}
Which rings $R$ have the property that every cyclic right $R$-module is weakly uniserial?
\end{Ques}

\section{\hspace{-6mm}. Torsion-free abelian groups of rank 1 which are weakly uniserial}

In this section, we answer  Question $2$.  In fact, it is determined  which torsion-free abelian groups of rank 1 are weakly uniserial. For this purpose, we need some definitions and results that will be presented below. A \textit{height sequence} $(\alpha_{p})_{p \in \pi}$   is a sequence of non-negative  integers  together with $\infty$, indexed by the elements of $\pi$ (the set of primes of  $\mathbb{Z}$). Let $A$ be a torsion-free  abelian group, $a \in A$  and $p$ is a prime number. Recall that  \textit{$p$-height}  $a$ in $A$, denoted by $h_{p}(a)$, is a non-negative integer $n$ with $a \in p^{n}A \backslash p^{n+1}A$ and $\infty$ if no such $n$ exists. Also   $h_{A}(a) $ (the height sequence  $a$ in $A$)  is the height sequence $(h_{p}(a))_{p \in \pi}$.
For simplicity, we write  $(\alpha_{p})$  instead of  the  height sequence  $(\alpha_{p})_{p \in \pi}$.  
Two height sequences $(\alpha_{p})$ and $(\beta_{p})$ are \textit{equivalent} if $\alpha_{p} = \beta_{p}$ for all but a finite number $p$ and $\alpha_{p} = \beta_{p}$ if either $\alpha_{p}= \infty$ or $\beta_{p}= \infty$. It is easy to see that  this relation is an equivalence relation. An equivalence class $\tau$ of height sequences is called a type, written $\tau=[ \alpha ] $ for some height sequence $\alpha$. Now if $A$ is a torsion-free  abelian group and $a \in A$, then we define the type  $a$ in $A$, to be ${\rm type}_{A}(a)=[h_{A}(a) ] $. If any two nonzero elements of  $A$ have the same type, then the common value being denoted by ${\rm type}(A)$. In this case $A$ is called \textit{homogeneous}, (for more details see \cite{Ar}).

\begin{Rem}\label{homogeneous} 
 \rm{ Every subgroup of  $ \mathbb{Q}$ is  homogeneous. To see this, let $A\leq \mathbb{Q}$ and  $a, b\in A$. Then there exist $m,n\in \mathbb{Z}$ such that $ma=nb$. It is easy to see that   $h_{p}(ma)= h_{p}(a)+  h_{p}(m)$ and  $h_{p}(nb)= h_{p}(b)+  h_{p}(n)$.  On the other hand,  $h_{p}(m)= h_{p}(n)=0$ for all but a finite number $p\in  \pi$.  Thus   ${\rm type}_{A}(a)=  [h_{A}(a) ]=[(h_{p}(a))_p]= [(h_{p}(a)+h_{p}(m))_{p}]= [(h_{p}(ma))_{p}]= [(h_{p}(nb))_{p}]= [(h_{p}(b)+h_{p}(n))_{p }]=[(h_{p}(b))_{p}]=[h_{A}(b) ]= {\rm type}_{A}(b)$.}

 \end{Rem}

\begin{Examp}
{\rm (1) ${\rm type}(\mathbb{Z}) = [h_{\mathbb{Z}}(1) ]= [(\alpha_{p})]$, where $\alpha_{p}=0$ for each prime number $p$.\\
(2) ${\rm type}(\mathbb{Q}) =  [h_{\mathbb{Q}}(1) ]= [(\beta_{p})]$, where $\beta_{p}=\infty$ for each prime number $p$.\\
(3) For $A=\langle \frac{1}{2^{n}}, \frac{1}{3^{m}} ~ | ~ n, m \in \mathbb{N} \rangle$, the subgroup of $\mathbb{Q}$ generated by $\frac{1}{2^{n}}$ and $\frac{1}{3^{m}}$ where $n, m \in \mathbb{N}$,  ${\rm type}(A)= [h_{A}(1) ]=[(\infty, \infty, 0, 0, 0, \ldots)]$.\\
(4) For the subgroup $A=\langle \frac{1}{2}, \frac{1}{3}, \frac{1}{5}, \ldots \rangle$ of $\mathbb{Q}$, we have ${\rm type}(A)= [h_{A}(1) ]= [(1, 1, 1, \ldots)]$.
}
\end{Examp}

The set of height sequences has a partial ordering given by $(\alpha_{p}) \leq (\beta_{p}) $ if $\alpha_{p} \leq \beta_{p}$ for each $p \in \pi$. Now let ${\rm type}(A)= [(\alpha_{p})]$ and ${\rm type}(B)= [(\beta_{p})]$, where $A$ and $B$ are homogeneous torsion-free abelian groups. Then we define ${\rm type}(A) \leq {\rm type}(B)$ if there exist $(\alpha'_{p}) \in [(\alpha_{p})]$ and $(\beta'_{p}) \in [(\beta_{p})]$ such that $(\alpha'_{p}) \leq (\beta'_{p})$. It is easy to see that if ${\rm type}(A) \leq {\rm type}(B)$ and ${\rm type}(B) \leq {\rm type}(A)$, then ${\rm type}(A) = {\rm type}(B)$. Let $A$ be a torsion-free abelian group. Recall that the \textit{rank} of $A$, denoted by  ${\rm rank}(A)$, is defined  ${\rm dim}_{\mathbb{Q}}(A \otimes_{\mathbb{Z}} \mathbb{Q})$.

\begin{Rem}\label{rank subgroup}
\rm{ A torsion-free abelian group $A$ has rank $1$ if and only if $A$ is isomorphic to a subgroup of $\mathbb{Q}$.
For if ${\rm rank}(A)=1$, then $A \otimes_{\mathbb{Z}} \mathbb{Q} \cong_{\mathbb{Q}} \mathbb{Q}$ and since $A$ is flat, $A$ is isomorphic to a subgroup of $\mathbb{Q}$. On the other hand, if $A$ is a subgroup of $\mathbb{Q}$, then since $\mathbb{Q}_{\mathbb{Z}}$ is flat, $A \otimes_{\mathbb{Z}} \mathbb{Q}$ is embedded  in $\mathbb{Q} \otimes_{\mathbb{Z}} \mathbb{Q} \cong_{\mathbb{Q}} \mathbb{Q}$. Thus ${\rm dim}_{\mathbb{Q}}(A \otimes_{\mathbb{Z}} \mathbb{Q}) =1$ and so ${\rm rank}(A)=1$.} 
 \end{Rem}

\begin{Lem}\label{etype}
{\rm{ (see \cite[Theorem 1.1]{Ar})}} Let $A$ and $B$ be torsion-free abelian groups of ${\rm rank}$ $1$. Then $A$  and $B$ are isomorphic  if and only if  ${\rm type}(A) = {\rm type}(B)$.
\end{Lem}

We note that if $\alpha=(\alpha_{p})_{p \in \pi}$ is a height sequence, then  there exists a torsion-free abelian group $A$ of rank $1$ with  ${\rm type}(A)=[\alpha]$, say $A= \langle \frac{1}{p^{n}} ~|~ p \in \pi, n=\alpha_{p} $ when~$ \alpha_{p}\neq \infty $~and~$  n\in  \mathbb{N} $~when~$  \alpha_{p}= \infty  \rangle$. Clearly ${\rm type}(A)=[h_{A}(1)] = [\alpha]$.

\begin{Lem}\label{ltype}
{\rm{(see \cite[Proposition 1.2]{Ar})}} Let $A$ and $B$ be torsion-free abelian groups of ${\rm rank}$ $1$. Then the following statements are equivalent:\\
\indent {\rm (a)} ${\rm Hom}_{\mathbb{Z}}(A, B) \neq 0$.\\
\indent {\rm (b)} ${\rm Hom}_{\mathbb{Z}}(A, B)$ contains a monomorphism.\\
\indent {\rm (c)} ${\rm type}(A) \leq {\rm type}(B)$.
\end{Lem}

In the following theorem,  we determine which torsion-free abelian groups of rank $1$  are weakly uniserial.

\begin{The}\label{t.f.a.g}
Let $A$ be a torsion-free abelian group of ${\rm rank}$ $1$. Then $A$ is weakly uniserial if and only if ${\rm type}(A)=[(\alpha_{p})]$, where $\alpha_{p}=0$ for all but a finite number $p$ and there is at most one $p$ such that $\alpha_{p}=\infty$.
\end{The}

\begin{proof}

 First assume that $A$ is weakly uniserial. By \Cref{rank subgroup},  we may assume that $A$ is a subgroup of $\mathbb{Q}$ and so, by \Cref{homogeneous}, we can consider the following cases:\\
\textbf{Case 1}: ${\rm type}(A)=[(\underbrace{0, 0,0, \ldots, 0}_{n-times}, m_{1}, m_{2}, m_{3}, \ldots )]$, where $n \in \mathbb{N} $ and $m_{i} \neq 0$ for any $i \geq 1$. Consider the following two height sequences:
\begin{center}
$\alpha=(\underbrace{0, 0,0, \ldots, 0}_{n-times}, m_{1}, 0, m_{3}, 0, m_{5}, \ldots )$, 
$\beta=(\underbrace{0, 0,0, \ldots, 0}_{n-times}, 0, m_{2}, 0, m_{4}, 0,  \ldots )$.
\end{center}
It is easy to see that $[\alpha] \nleq [\beta]$ and $[\beta] \nleq [\alpha]$. Thus if $G_{1}$ and $G_{2}$ are subgroups of $A$, such that ${\rm type}(G_{1})=[\alpha]$ and ${\rm type}(G_{2})=[\beta]$, then by \Cref{ltype}, ${\rm Hom}_{\mathbb{Z}}(G_{1}, G_{2}) =0= {\rm Hom}_{\mathbb{Z}}(G_{2}, G_{1})  $. Therefore $A$ is not weakly uniserial, a contradiction.\\
\textbf{Case 2}:  ${\rm type}(A)=[ (m_{1}, m_{2}, \ldots, m_{n}, 0, 0, 0, \ldots )]$, where $n \in \mathbb{N} $. If there exist $1 \leq i,j \leq n$ with $i < j$ and $m_{i}=m_{j}=\infty$, then we consider the following height sequences:
\begin{center}
$\alpha= (m_{1}, m_{2} \ldots, m_{i-1}, \infty, m_{i+1}, \ldots,  m_{j-1}, 0,  m_{j+1}, \ldots, m_{n}, 0, 0, 0, \ldots)$,
$\beta= (m_{1}, m_{2} \ldots, m_{i-1}, 0, m_{i+1}, \ldots,  m_{j-1}, \infty,  m_{j+1}, \ldots, m_{n}, 0, 0, 0, \ldots)$.
\end{center}
 Clearly 
$[\alpha] \nleq [\beta]$ and $[\beta] \nleq [\alpha]$. Now suppose that $G_{1}$ and $G_{2}$ are subgroups of $A$ such that ${\rm type}(G_{1})=[\alpha]$ and ${\rm type}(G_{2})=[\beta]$. Then by \Cref{ltype}, ${\rm Hom}_{\mathbb{Z}}(G_{1}, G_{2}) =0= {\rm Hom}_{\mathbb{Z}}(G_{2}, G_{1})  $ and so $A$ is not weakly uniserial, a contradiction.  Therefore,  at most one  of the $m_i$'s is equal to $\infty$, as desired.\\
Conversely, suppose that  ${\rm type}(A)=[ (m_{1}, m_{2}, \ldots, m_{n}, 0, 0, 0, \ldots )]$, where $n \in \mathbb{N}$, $m_i\in \mathbb{ N}\cup \{\infty\}$ and at most one  of the $m_i$'s is equal to $\infty$. We show that $A$ is weakly uniserial. Let $G_1$ and $G_2$ be two subgroups of $A$. Clearly,    ${\rm Hom}_{\mathbb{Z}}(G_{1}, A)\neq 0$ and  ${\rm Hom}_{\mathbb{Z}}(G_{2}, A)\neq 0$. Thus by  \Cref{ltype}, ${\rm type}(G_1) \leq {\rm type}(A)$ and ${\rm type}(G_2) \leq {\rm type}(A)$ and we may assume   ${\rm type}(G_1)=[ (k_{1}, k_{2}, \ldots, k_{n}, 0, 0, 0, \ldots )]$ and  ${\rm type}(G_2)=[ (l_{1}, l_{2}, \ldots, l_{n}, 0, 0, 0, \ldots )]$, where $k_i\leq m_i$ and $l_i\leq m_i$ for $1\leq i \leq n$. Consider the following cases:\\
\textbf{Case 1}: $m_i\in  \mathbb{N} $,  for $1\leq i \leq n$. Then it is clear that ${\rm type}(G_1) = {\rm type}(G_2)$ and so by   \Cref{etype}, $G_1\cong G_2.$\\
\textbf{Case 2}: Without loss of generality, we may assume that $m_1=\infty$. In this case, either $k_1=l_1=\infty$ or $k_1, l_1\in  \mathbb{N}$ or only one of them is equal to $\infty$. Therefore either ${\rm type}(G_1) = {\rm type}(G_2)$ or ${\rm type}(G_1) \leq {\rm type}(G_2)$ or ${\rm type}(G_2) \leq {\rm type}(G_1)$. Then  by   \Cref{etype} and \Cref{ltype},   either $G_1\cong G_2$ or one of them is embedded in the other.  
\end{proof}

 Now, by \Cref{t.f.a.g}, we have the following example.

\begin{Examp}
{\rm (1) $\mathbb{Q}_{\mathbb{Z}}$ is not weakly uniserial, because ${\rm type}(\mathbb{Q}) = [(\infty, \infty, \infty, \ldots)]$, see also \Cref{Q not weak}.

(2)  $\mathbb{Z}_{\mathbb{Z}}$ is  weakly uniserial, because ${\rm type}(\mathbb{Z}) = [(0, 0, 0, \ldots)]$.

(3) If $A= \langle \frac{1}{2^{n}}, \frac{1}{3}, \frac{1}{5} ~|~ n \in \mathbb{N} \rangle$, then ${\rm type}(A)=[(\infty, 1, 1, 0, 0, 0, \ldots)]$. Thus $A$ is a weakly uniserial group.

(4) If $A= \langle \frac{1}{2^{n}}, \frac{1}{3^{m}}, \frac{1}{5} ~|~ n, m \in \mathbb{N} \rangle$, then ${\rm type}(A)=[(\infty, \infty, 1, 0, 0, 0, \ldots)]$. Thus  $A$ is not weakly uniserial.

(5)  If $A= \langle \frac{1}{2}, \frac{1}{3}, \frac{1}{5}, \ldots \rangle$, then ${\rm type}(A)=[(1, 1, 1, \ldots)]$. Thus $A$ is not  weakly uniserial.
}
\end{Examp}

\section{\hspace{-6mm}. Finitely generated weakly uniserial modules over commutative principal ideal domains }

In this  section,  we completely determine the structure of  finitely generated weakly uniserial $R$-modules, where $R$ is a commutative principal ideal domain.  In particular, it is shown that a finitely generated $\mathbb{Z}$-module $M$ is weakly uniserial if and only if $M \cong \mathbb{Z}^{n} $ or $M \cong \oplus_{n} \mathbb{Z}_{p}$ or $M \cong \mathbb{Z}_{p^{n}}$, where $p$ is a prime number and $n \geq 0$ is an integer. Also the structure of weakly uniserial $\mathbb{Z}$-modules with nonzero socle is given. \vspace{1.6mm}\\
Let $x$ be an element of a $\Bbb{Z}$-module $M$. The {\it  order} of $x$ is the smallest integer $n$ such that $nx=0$, if no such $n$ exists, the order $x$ is called  {\it infinite}. The order  of $x$ is denoted by ${\rm ord}(x)$.

\begin{Pro}\label{finite abelian}
Let $M$ be a finite right $R$-module and $|M|=p_{1} ^{\alpha_{1}} \ldots  p_{t}^{\alpha_{t}}$, where $t\geq2$,   $\alpha_{i} \in\Bbb{N}$ and $p_i$'s are distinct prime numbers for each $ 1\leq i \leq t$. Then $M$ is not weakly uniserial.
\end{Pro}
\begin{proof}
By fundamental theorem of finite abelian groups we have $M=M_{1}\oplus \cdots \oplus M_{t}$, where $|M_{i}| = p_{i}^{\alpha_{i}}$, $\alpha_{i} \in \Bbb{N}$ and $p_{i}$'s are distinct prime numbers for each $ 1\leq i \leq t$.
First we show that every $M_{i}$ is a right $R$-submodule of $M$. Assume that $s \in \{1, \ldots, t \}$ and $y \in M_{s}$. Then for every $r\in R$ we have $yr=y_{1}+ \cdots +y_{s}+ \cdots +y_{t}$, where $y_{i}\in M_{i}$ for each $ 1\leq i \leq t $. Hence $p_{s} ^{\alpha_{s}} yr =p_{s} ^{\alpha_{s}} y_{1}+\cdots
+p_{s} ^{\alpha_{s}}y_{s} + \cdots +p_{s} ^{\alpha_{s}}y_{t}$ and so $p_{s} ^{\alpha_{s}}y_{1}+\cdots+p_{s} ^{\alpha_{s}}y_{s-1} + p_{s} ^{\alpha_{s}}y_{s+1}+ \cdots + p_{s} ^{\alpha_{s}}y_{t}=0$. It follows that  $p_{s} ^{\alpha_{s}}y_{1}=\cdots=p_{s} ^{\alpha_{s}}y_{s-1}=p_{s} ^{\alpha_{s}}y_{s+1}= \cdots = p_{s} ^{\alpha_{s}}y_{t}=0$. On the other hand for every $ 0 \neq  y_{i} $, we have ${\rm ord}(y_{i})=p_{i}^{\beta_{i}}$ where  $ 1 \leq\beta_{i}\leq \alpha_{i}$.  Thus  for every $i \neq s$,  we conclude that $p_{i}^{\beta_{i}} \mid p_{s} ^{\alpha_{s}} $ and so $p_{i} \mid p_{s}$, a contradiction. Therefore   for every $i \neq s$, $y_{i}= 0$ and so  $ yr = y_{s} \in M_{s}$. This means that  $M_{s}$ is an $R$-submodule of $M$. Now, if $M$ is a weakly uniserial right $R$-module, then $M_{i} \rightarrowtail M_{j}$ or $M_{j} \rightarrowtail M_{i}$, where $i \neq j$ and $ 1\leq i, j \leq t $. Therefore, $p_{i} \mid p_{j}$ or $p_{j} \mid p_{i}$, a contradiction.
\end{proof}

\begin{Cor}\label{cor of finite}
If $M$ is a finite weakly uniserial right $R$-module, then $|M| = p^{\alpha} $ where $p$ is a prime number and $\alpha \geq 0 $ is an integer.
\end{Cor}

\begin{Rem}\label{rem of finite}
{\rm The converse of \Cref{cor of finite} is not necessarily true. For example suppose that $M \cong \mathbb{Z} _{p^{\alpha_{1}}} \oplus \mathbb{Z}_{p^{\alpha_{2}}} \oplus \cdots \oplus \mathbb{Z}_{p^{\alpha_{t}}}$  as $\mathbb{Z}$-modules, where $p$ is a prime number and $\alpha_{i} \in \mathbb{N}$, for each $ 1 \leq i \leq t$. Assume that there exists $s\in \{1, \ldots, t \}$ such that $\alpha _{s} \geq 2$. Then
$M_{1} = (0) \oplus \cdots \oplus (0) \oplus \mathbb{Z}_{p^{\alpha_{s}}} \oplus (0) \oplus \cdots \oplus (0)$ and $M_{2} = p^{\alpha_{1}-1} \mathbb{Z}_{p^{\alpha_{1}}} \oplus (0) \oplus \cdots \oplus (0) \oplus p \mathbb{Z}_{p^{\alpha_{s}}} \oplus (0) \oplus \cdots \oplus (0)$ are  submodules of $M$ that $|M_{1}| = |M_{2}|$, but since $M_{1}$ is cyclic and $ M_{2}$ is not cyclic we conclude that $M_{1} \ncong M_{2}$. Therefore, $M_{1} \not\rightarrowtail M_{2}$ and $M_{2} \not\rightarrowtail M_{1}$.}
\end{Rem}

It is well-known that  being uniserial is not a symmetric property. For example if $R$ is the left skew polynomial ring $ K[x;\sigma] \slash (x^{2})$, where $K$ is a field and $\sigma : K \rightarrow K $ is a ring homomorphism which is not an automorphism, then $R$ is a left uniserial ring that is not right uniserial. In \cite[Lemma 1]{Clark}, W. E. Clark and D. A. Drake showed that if $R$ is a finite ring, then $R$ is right uniserial if and only if it is left uniserial. We could not show that being weakly uniserial is a symmetric property, but the following result shows that under certain conditions a right weakly uniserial ring is also a left weakly uniserial ring.

\begin{Pro}\label{finite ring}
Let $R$ be a finite ring with a principal maximal right ideal. If $R$ is a right weakly uniserial ring, then either $R \cong M_{n}(D)$ where $D$ is a field or $R$ is left and right uniserial.
\end{Pro}

\begin{proof}
Since $R$ is a right Artinian right weakly uniserial ring, by \Cref{right artinian right weakly}(b), $R$ is local or $R \cong M_{n}(D)$ where $D$ is a field. If $R \cong M_{n}(D)$, then we are done. So assume that $R$ is a local ring and $\M$ is the maximal right ideal of $R$. By hypothesis there exists $x \in R$ such that $\M = xR$. Note that $x \notin \M^{2}$, for if not, we conclude that $\M = \M^{2}$ and since $\M$ is nilpotent we see that $\M = 0$, a contradiction. Since $R$ is a finite ring, $R \slash \M$ is a field and hence $\M \slash \M^{2}$ is a right $R \slash \M$-vector space. Thus $\M \slash \M^{2} \cong R \slash \M$ and so $|\M \slash \M^{2}| = |R \slash \M |$. Note that $\M$ is a two-sided ideal, so that by the following definition  $\M \slash \M^{2}$ is a left $R \slash \M$-vector space
\begin{center}
$(r + \M) (y +\M^{2}) = ry +\M^{2}$,
\end{center}
for any $r \in R$ and any $y \in\M$. Suppose that ${\rm dim}_{R \slash \M}(\M \slash \M^{2})=n < \infty$. Then $|\M \slash \M^{2}| = |R \slash \M |^{n}$ and so $n=1$. Hence $\M \slash \M^{2}$ is a left $R \slash \M$-vector space of dimension $1$ and so $\M=Rx+\M^{2}$. Let $t$ be the index of nilpotency of $\M$. Then $\M=Rx+\M^{2}=Rx+\M^{3}= \ldots = Rx+\M^{t}=Rx$. Thus $\M=xR=Rx$ and $\M^{n}=x^{n}R=Rx^{n}$, for any $n \in \mathbb{N}$. Now by Nakayama's Lemma we have the following proper chain
\begin{center}
$0=x^{t}R \subsetneq x^{t-1}R \subsetneq \cdots \subsetneq x^{2}R \subsetneq xR \subsetneq R$.
\end{center}
Assume that $0 \neq a \in R$ and $i$ is the smallest number that $a \in x^{i}R  \backslash x^{i+1}R$. If $a$ is unit, then $aR=R$. If $a$ is not unit, then $a \in \M=xR$ and since $a \in x^{i}R  \backslash x^{i+1}R$, then $aR \subseteq  x^{i}R $. We show that $x^{i}R \subseteq aR$. Since $aR \subseteq  x^{i}R $, there exists $r \in R$ such that $a= x^{i} r$. If $r$ is unit, then $x^{i}=ar^{-1} \in aR$ and hence $x^{i}R \subseteq aR$. If $r$ is not unit, then $r \in \M=xR$ and so $a=x^{i} r \in x^{i}RxR =x^{i+1}R$, a contradiction. Hence $aR=x^{i}R$ and we conclude that $R$ is a right uniserial ring. Now by \cite[Lemma 1]{Clark}, $R$ is also a left uniserial ring. 
\end{proof}

\begin{Lem}\label{torsion and torsion-free}
If $M$ is a weakly uniserial right $R$-module, then  $M$  does not include nonzero torsion submodules and torsion-free submodules simultaneously.
\end{Lem}
\begin{proof}
Assume that $N$ and $K$ are nonzero torsion-free and torsion submodules of $M$, respectively. Since $M$ is weakly uniserial, $N \rightarrowtail K$ or $ K \rightarrowtail N$. If $N \overset{f} \rightarrowtail K$, then for every $n \in N$ there exists a regular element $r$ of $R$ such that $0=f(n)r=f(nr)$. Hence $nr=0$ and so $n=0$, a contradiction. If $K \overset{f}\rightarrowtail N$, then for every $k \in K$ there exists a regular element $r$ in $R$ such that $kr=0$. Hence $f(k)r=f(kr)=0$, and so $f(k)=0$. This means that  $k=0$, a contradiction.
\end{proof}

\begin{The}\label{fundamental}
Let $R$ be a commutative principal ideal domain and $M$ be a finitely generated $R$-module. Then $M$ is weakly uniserial if and only if $M \cong R^{n}$ or $M \cong \oplus_{n} (R/{pR})$ or $M \cong R/{p^{n}R}$, where $p$ is a prime element of $R$ and $n \geq 0$ is an integer.
\end{The}
\begin{proof}
Assume that $M$ is a  nonzero  weakly uniserial $R$-module. By the fundamental structure theorem for finitely generated modules over a principal ideal domain, $M \cong R^{n} \oplus R/{p_{1}^{\alpha_{1}}R} \oplus \cdots \oplus R/{p_{r}^{\alpha_{r}}R} $, where $n \geq 0$ is an integer, $\alpha_i\geq 0$ for $1\leq i \leq r$ and  each $ p_i$  is a prime element in $R$. Since
$R$ is torsion-free and for any  $i$, $R/{p_{i}^{\alpha_{i}}R}$ is torsion, by \Cref{torsion and torsion-free}, this decomposition does not include $R$ and $R/{p_{i}^{\alpha_{i}}R}$  simultaneously. 
 Thus $M \cong R^{n}$ or $M \cong  R/{p_{1}^{\alpha_{1}}R} \oplus \cdots \oplus R/{p_{r}^{\alpha_{r}}R}  $. Suppose that $M \cong  R/{p_{1}^{\alpha_{1}}R} \oplus \cdots \oplus R/{p_{r}^{\alpha_{r}}R}  $. We proceed by cases:\\
\textbf{Case 1:} There exist $ p_i, p_j\in \{ p_1, \ldots, p_r\}$ such that $p_i$ and $p_j$ are not associates. In this case we consider $R$-module   $R/{p_i^{\alpha_{i}}R} \oplus R/{p_j^{\alpha_{j}}R} $. Since $M$ is weakly uniserial, without loss of generality, we can assume that  $R/{p_i^{\alpha_{i}}R} \rightarrowtail   R/{p_j^{\alpha_{j}}R} $. It follows that  ${\rm Ann}_R(R/{p_j^{\alpha_{j}}R}) \subseteq {\rm Ann}_R(R/{p_i^{\alpha_{i}}R})$ and so $p_j^{\alpha_j}=p_i^{\alpha_i} r $, for some $r\in R$. Thus $p_i |p_j^{\alpha_j}$ and hence 
$p_i |p_j$, i.e.,  $p_i$ and $p_j$ are  associates, a contradiction.\\
\textbf{Case 2:} For every $1 \leq i,j \leq r$,   $p_i$ and $p_j$ are associates. In this case we can write $M \cong  R/p^{\alpha_1}R \oplus \cdots \oplus R/p^{\alpha_r}R$, for some prime element $p$ of $R$. We show that if $r>1$, then  for any  $1 \leq i \leq r$, $\alpha_i=1$. Suppose that $\alpha_i  > \alpha_j\geq 1$, for some  $1 \leq i,j \leq r$.  It is easy to check that   $ R/p^{\alpha_i}R$ is uniform (uniserial). Thus $(p^{\alpha_{i}-1}R/p^{\alpha_{i}}R)  \oplus (p^{\alpha_{j}-1} R/p^{\alpha_{j}}R) \not\rightarrowtail   R/p^{\alpha_{i}}R $. Since $M$ is weakly uniserial, we must have  $    R/p^{\alpha_{i}}R     \rightarrowtail    (p^{\alpha_{i}-1}R/p^{\alpha_{i}}R)  \oplus (p^{\alpha_{j}-1} R/p^{\alpha_{j}}R) $. But it follows that $p \in {\rm Ann}_R (R/p^{\alpha_i}R )$, a contradiction. Therefore for any $i$, $\alpha_i=1$ and the proof of one direction is complete.\\
Convesely, if   $M \cong R/{p^{n}R}$, then $M$ is  uniserial and so it is weakly uniserial.  If $M \cong \oplus_{n}  (R/{pR})$, then by \Cref{Some exam}(1), $M$ is weakly uniserial. Finally, if $M \cong R^{n}$, then by \Cref{domain}(c), $M$ is weakly uniserial.
\end{proof}

\begin{Cor}
A finitely generated $\mathbb{Z}$-module $M$ is weakly uniserial if and only if $M \cong \mathbb{Z}^{n} $ or $M \cong \oplus_{n} \mathbb{Z}_{p}$ or $M \cong \mathbb{Z}_{p^{n}}$, where $p$ is a prime number and $n \geq 0$ is an integer.
\end{Cor}

We conclude the paper with the following result, which gives the structure of weakly uniserial $\mathbb{Z}$-modules with nonzero socle.

\begin{Pro}\label{w.uni z-module}
Every $\mathbb{Z}$-module $M$ is weakly uniserial with ${\rm Soc}(M) \neq 0$ if and only if $M \cong \mathbb{Z}_{p^{n}}$ or $M \cong \mathbb{Z}_{p^{\infty}}$ or $M \cong \oplus_{I} \mathbb{Z}_{p} $, where $p$ is a prime number, $n \in \mathbb{N} $, and $I$ is an index set.
\end{Pro}

\begin{proof}
Suppose that $M$ is a weakly uniserial $\mathbb{Z}$-module with ${\rm Soc}(M) \neq 0$. Then by \Cref{essential socle}(a), ${\rm Soc}(M) \subseteq_{e} M$. If ${\rm Soc}(M)$ is simple, then $M$ is cocyclic and so by \cite[Theorem 2.6]{Yah}, there exists a prime number $p$ such that $M \rightarrowtail {\rm Hom}_{\mathbb{Z}}( \mathbb{Z} , \mathbb{Z}_{p^{\infty}} ) \cong \mathbb{Z}_{p^{\infty}}$. Hence $M \cong \mathbb{Z}_{p^{\infty}}$ or $M \cong \mathbb{Z}_{p^{n}}$ for some $n \geq 0$. Now assume that ${\rm Soc}(M)$ is not simple and $S$ is a simple submodule of $M$. Then $S = x \mathbb{Z}$ for some $x \in S$. If ${\rm ord}(x)=\infty$, then $x\mathbb{Z}\cong \mathbb{Z} $, a contradiction. Thus ${\rm ord}(x) < \infty$ and we conclude that $| x \mathbb{Z}| = |S| =p$ for some prime number $p$. Now we claim that ${\rm ord}(m)=p$ for every $0 \neq m \in M$. If there exists $m \in M$ such that ${\rm ord}(m)=\infty$, then $m\mathbb{Z}\cong \mathbb{Z} $ and since $M$ is weakly uniserial, $m\mathbb{Z} \rightarrowtail S$ or $S \rightarrowtail m\mathbb{Z}$, which in any case we have a contradiction. Thus ${\rm ord}(m)<\infty$ and so $m\mathbb{Z}$ is a finite weakly uniserial $\mathbb{Z}$-module. By \Cref{cor of finite}, $|m\mathbb{Z}|=q^{n}$ where $q$ is a prime number and $n \in \mathbb{N} $. We show that $n=1$ and $q=p$. If $n \geq 2$, then by the first Sylow's Theorem $m\mathbb{Z}$ contains a subgroup $N$ of order $q^{2}$. Since ${\rm Soc}(M)$ is not simple, there exists $S_{1}\oplus S_{2} \subseteq {\rm Soc}(M)$, where $S_{1}$ and $S_{2}$ are  isomorphic simple submodules and $|S_{1}|=|S_{2}|=p$. Thus $S_{1}\oplus S_{2} \rightarrowtail N$ or $N \rightarrowtail S_{1}\oplus S_{2} $, a contradiction. Hence ${\rm ord}(m)=q$ for every $0 \neq m \in M$. Again since $M$ is weakly uniserial, $m\mathbb{Z} \rightarrowtail S$ or $S \rightarrowtail m\mathbb{Z}$ and so $p=q$. Therefore, $M$ is a vector space over $\mathbb{Z}_{p}$ by multiplication $\overline{c}x=cx$ for every $\overline{c} \in \mathbb{Z}_{p}$ and any $x \in M$. Then $M \cong \oplus_{I} \mathbb{Z}_{p}$ where $I$ is an index set. The converse is clear.
\end{proof}
{\bf Acknowledgment.}  The research of the third author was in part supported by a  grant from IPM (No. 1400160414). This research is partially carried out in the IPM-Isfahan Branch.

\end{document}